\documentclass[12pt,a4paper, reqno]{amsart}
\usepackage[abbrev]{amsrefs}
\usepackage{amssymb,mathrsfs}
\usepackage{xcolor}
\usepackage{bm}
\usepackage{mathtools}

\usepackage{etoolbox}

\makeatletter
\patchcmd{\@settitle}{\uppercasenonmath\@title}{}{}{}
\patchcmd{\@setauthors}{\MakeUppercase}{}{}{}
\patchcmd{\section}{\scshape}{}{}{}
\makeatother

\theoremstyle{plain}
 \newtheorem{theorem}{Theorem}[section]
 \newtheorem{lemma}[theorem]{Lemma}
 \newtheorem{corollary}[theorem]{Corollary}
 \newtheorem{proposition}[theorem]{Proposition}

\theoremstyle{definition}
 \newtheorem{definition}[theorem]{Definition}
 \newtheorem{condition}[theorem]{Condition}
 
\theoremstyle{remark}
 \newtheorem{remark}[theorem]{Remark}
 
 \newtheorem*{ack}{Acknowledgment}

\def\KL{\mathop{D_\mathrm{KL}}\nolimits}
\def\tr{\mathop{\mathrm{tr}}\nolimits}
\numberwithin{equation}{section}
\begin{document}
\title[]{\large Divergence and Deformed Exponential Family}
\author[]{Hiroshi Matsuzoe $^{\flat}$}
\author[]{Asuka Takatsu$^{\dagger}$}
\thanks{$\flat$
Department of Computer Science, Nagoya Institute of Technology, Nagoya 466-8555, Japan\\
\qquad{\sf matsuzoe@nitech.ac.jp}}
\thanks{$\dagger$
Graduate School of Mathematical Sciences, The University of Tokyo,
3-8-1 Komaba, Meguro-ku, Tokyo 153-8914, Japan;
 RIKEN Center for Advanced Intelligence Project (AIP), 
Nihonbashi 1-chome Mitsui Building, 15th floor,
1-4-1 Nihonbashi, Chuo-ku, Tokyo 103-0027, Japan\\
\qquad{\sf asuka-takatsu@g.ecc.u-tokyo.ac.jp}}
\maketitle
\vspace{-15pt}
\begin{abstract}
The Kullback--Leibler divergence together with exponential families establishes the foundation of information geometry and is widely generalized.
Among the generalization,   we focus on the $(h,\tau)$-divergence and $(h,\tau)$-exponential families.
We present a sufficient condition for the $(h,\tau)$-divergence to induce a Hessian structure on an $(h,\tau)$-exponential family.
We also define the $(h,\tau)$-dependence of random variables and prove a kind of the law of large numbers.
\end{abstract}
\tableofcontents
\vspace{-30pt}
\section{Introduction}\label{intro}
For a Radon measure~$\mu$ on a topological space ${\mathcal{X}}$,
set 
\begin{align*}
\mathcal{P}( \mu)
&\coloneqq \{p\in L^1(\mu) \mid p>0 \text{ $\mu$-a.e. and } \|p\|_{L^1(\mu)}=1 \}.
\end{align*}
The \emph{Kullback--Leibler divergence} $\KL\colon \mathcal{P}(\mu) \times \mathcal{P}(\mu)\to [0,\infty]$ is defined by 
\begin{align*}
\KL(p,p')\coloneqq \int_{{\mathcal{X}}} p(x) \log \frac{p(x)}{p'(x)}\, d\mu(x),
\end{align*}
which is well-defined due to the following nonnegativity 
\[
t \log \frac{t}{s}-(t-s)
=t\log t -s\log s- (t-s) (\log s+1) \geq 0
\quad \text{for }s,t>0. 
\]
The inequality also ensures that $\KL(p,p')= 0$ holds  if and only if $p=p'$ $\mu$-a.e.
Thus the Kullback--Leibler divergence expresses the difference between two positive probability densities.
The Kullback--Leibler divergence has enormous amount of generalizations with different roots.
%
In information geometry, a generalization of the Kullback--Leibler divergence is called a \emph{divergence} or \emph{contrast function}.
%
\begin{definition}
A \emph{divergence} $D$ on a set $\mathcal{M}$ is a function on $\mathcal{M} \times \mathcal{M}$ valued in $[0,\infty]$ such that $D(p,p')\geq 0$ with equality  if and only if $p=p'$.
A \emph{contrast function} is a smooth divergence on a smooth manifold valued in $[0,\infty)$.
\end{definition}
The Kullback--Leibler divergence is indeed a divergence on $ \mathcal{P}(\mu)$.
If ${\mathcal{X}}$ is a finite set and $\mu$ is the counting measure on ${\mathcal{X}}$, 
then $\mathcal{P}(\mu)$ is identified with an open simplex and 
the Kullback--Leibler divergence becomes a contrast function on $\mathcal{P}(\mu)$.
A manifold in $\mathcal{P}(\mu)$ is often referred to as a \emph{statistical model}, which  is well-established in statistics.
However, we impose additional regularity conditions  on a statistical model (see Definition~\ref{compati}).

In information geometry, the relation between a contrast function and a Hessian structure is emphasized.
A \emph{Hessian structure} on a manifold $\mathcal{M}$ is a pair of a Riemannian metric $g$ and a flat connection $\nabla$
such that  $g=\nabla d \Psi$ holds locally for some  smooth function $\Psi$ on $\mathcal{M}$.
In this case, $\Psi$ is called a \emph{local potential}.
The Hessian structure  induces a local contrast function, called the \emph{canonical divergence}  (see Proposition~\ref{cano}).
The canonical divergence $D$ satisfies the following condition.
\begin{condition}\label{Eguchi}
For $V\in \mathfrak{X}(\mathcal{M})$,
set $\widetilde{V}\coloneqq (V,0),\widehat{V}\coloneqq (0,V)\in \mathfrak{X}(\mathcal{M}\times \mathcal{M})$.
Then 
\[
-\widetilde{V}_{(p,p)}(\widehat{V}D)>0\quad\text{for $V\in \mathfrak{X}(\mathcal{M})$ and $p\in \mathcal{M}$ with $V_p\neq0$.}
\]
\end{condition}
Conversely, if a contrast function $D$ on a manifold $\mathcal{M}$ satisfies Condition~\ref{Eguchi},
then two maps $g\colon \mathfrak{X}(\mathcal{M}) \times \mathfrak{X}(\mathcal{M}) \to C^\infty(\mathcal{M})$,
$\nabla\colon\mathfrak{X}(\mathcal{M}) \times \mathfrak{X}(\mathcal{M}) \to \mathfrak{X}(\mathcal{M})$
defined by 
\begin{equation}
\label{met}
g_p(V,W)\coloneqq -\widetilde{V}_{(p,p)} (\widehat{W}D),\quad
g_p(\nabla_ZV, W)\coloneqq -\widehat{Z}_{(p,p)} (\widehat{V}\widetilde{W} D),
\end{equation}
become a Riemannian metric and a connection, respectively 
(see for example \cite{Eguchi92}*{p.633}).
However, the pair $(g,\nabla)$ is not necessarily a Hessian structure on $\mathcal{M}$.

In this paper, we focus on the $(h,\tau)$-divergence and $(h,\tau)$-exponential families among contrast functions and statistical models and study its mathematical structure.
Here $\tau$ is a smooch function on an open interval $I$ contained in $(0,\infty)$
and $h$ is a smooth function on $\tau(I)$ such that $\tau'>0$ on $I$ and $h''>0$ on $\tau(I)$.
We denote by $\mathcal{G}$ the set of such triples~$(h,\tau, I)$.

For $(h,\tau,I)\in \mathcal{G}$, define 
\[
d_{h,\tau}(t,s)\coloneqq h(\tau(t))-h(\tau(s))- \left(\tau(t)-\tau(s)\right) h'(\tau(s)) 
\quad \text{for }s,t\in I,
\]
which is nonnegative by the positivity of  $h''$ on $\tau(I)$.
We set
\begin{align*}
\mathcal{P}_I(\mu)&\coloneqq \left\{ p\in \mathcal{P}(\mu) \bigm| \text{$p\in I\ \mu$-a.e.}\right\}
\end{align*}
and define the \emph{$(h,\tau)$-divergence} $D_{h,\tau}\colon \mathcal{P}_I(\mu)\times \mathcal{P}_I(\mu) \to [0,\infty]$ by 
\begin{equation}\label{div1}
D_{h,\tau}(p,p')\coloneqq \int_{{\mathcal{X}}} d_{h,\tau}(p(x), p'(x)) d\mu(x).
\end{equation}
A notion of $(h, \tau)$-divergence is originally introduced in~\cite{Zhang}*{Proposition~6} and 
the expression~\eqref{div1} was provided in~\cite{NZ}*{Proposition~1}.
Under  an appropriate condition such as the compactness of $\mathcal{X}$,
the $(h,\tau)$-divergence on a statistical model in $\mathcal{P}_I(\mu)$ induces a Riemannian metric $g^{h,\tau}$ and a connection $\nabla^{h,\tau}$ via~\eqref{met}
(see Lemma~\ref{metcon}).

For $(h,\tau, I) \in \mathcal{G}$,
a statistical model $\mathcal{M}$ in $\mathcal{P}(\mu)$ is an \emph{$(h,\tau)$-exponential family}
if, for each point in $\mathcal{M}$, there exist a coordinate system $(U,\theta)$ around the point, 
a continuous map $T\colon \mathcal{X}\to \mathbb{R}^{n}$, 
a continuous function $c$ on $\mathcal{X}$
and a smooth function $\psi$ on $\mathcal{M}$
 such that 
\[
h'(\tau(p(x)))=\langle \theta(p), T(x) \rangle-c(x)-\psi(p) 
\quad 
\text{for }p\in U \text{ and }x\in{\mathcal{X}}.
\]
We call $(U,\theta, T, c,\psi)$ an \emph{$(h,\tau)$-exponential coordinate representation} 
and $\psi$ a local \emph{$(h,\tau)$-exponential normalization}, respectively
 (see Definition~\ref{exp}).
The terminology ``$(h,\tau)$-exponential" comes from the fact that the  extended inverse function of $h' \circ \tau$
is denoted by $\exp_{h,\tau}$ and called the \emph{$(h,\tau)$-exponential function}.
Note that  $\psi$ is a normalization hence it is determined by the other ingredients $(\mathcal{M},\theta, T, c)$.
An $(h,\tau)$-exponential family has a unique global $(h,\tau)$-exponential coordinate representation up to affine transformation 
(see Propositions~\ref{minimal} and~\ref{global}).
%
For example, the $(h,\tau)$-divergence on $\mathcal{P}_I(\mu)$ with the choice
\begin{equation*}
h(r)\coloneqq  r\log r,  \qquad 
\tau(t)\coloneqq t, \qquad  
I\subset (0,\infty)
\end{equation*}
coincides with the Kullback--Leibler divergence on $\mathcal{P}_I(\mu)$, where 
\[
\exp_{h,\tau}(u)=e^{u-1}\quad \text{for }u\in \mathbb{R}. 
\]
In this case, the terminology $(h,\tau)$-exponential is abbreviated to merely \emph{exponential}.

Our first result is to give a sufficient condition for $( g^{h,\tau},\nabla^{h,\tau})$ being a Hessian structure on an $(h,\tau)$-exponential family.
Note that if  $\mathcal{X}$ is compact, 
then, for $\phi\in C^\infty(I)$, 
\[
\mathbb{I}_{\phi}(p)\coloneqq \int_{\mathcal{X}} \phi(p(x)) d\mu(x)
\]
determines a smooth function on an $(h,\tau)$-exponential family (see Remark~\ref{remm}\eqref{remm2}).
\begin{theorem}\label{mainthm2}
Let $(h, \tau, I)\in \mathcal{G}$ and $\mathcal{M}$ be an $(h,\tau)$-exponential family in~$\mathcal{P}_I(\mu)$.
If  $\mathcal{X}$ is compact and~$\mathbb{I}_\tau$ is constant on $\mathcal{M}$, 
then $( g^{h,\tau},\nabla^{h,\tau})$ is a Hessian structure on $\mathcal{M}$ 
and its  global potential is $-\mathbb{I}_{s^\star_{h,\tau}}+\psi\mathbb{I}_\tau$, 
where
$\psi$ is a global $(h,\tau)$-exponential normalization of $\mathcal{M}$ and 
 $s^{\star}_{h,\tau}\colon I\to \mathbb{R}$ is defined~by 
\[
s^{\star}_{h,\tau}\coloneqq-\tau \cdot (h'\circ \tau)+(h\circ \tau).
\]
\end{theorem}
The compactness of ${\mathcal{X}}$ is removed by requiring a certain regularity 
(see Theorem~\ref{mainthm22}). 

%
An $(h,\tau)$-exponential family is characterized as the set of constrained minimizers
by the Pythagorean relation for the $(h,\tau)$-divergence (see Proposition~\ref{maxcor}).
Although the Pythagorean relation on a statistical model is often referred to as a by-product of a Hessian structure (see for example~\cite{Amari}*{Chapter~1.6.1}),
this relies only on the existence of a global coordinate system and is extended to a suitable subset of~$\mathcal{P}(\mu)$.
For example, the extension of the Pythagorean relation for  the Kullback--Leibler divergence is found in ~\cite{CNN}*{Theorem~22.1},
which is used in a dual problem of the maximum likelihood estimation on an exponential family.

The \emph{maximum likelihood estimation} is a method to 
infer the density $p_\ast\in \mathcal{P}(\mu)$ of law of an unknown ${\mathcal{X}}$-valued random variable from given data $\bm{x}_\ast \in \mathcal{X}^k$, where $k\in \mathbb{N}$.
Let $\mu^{\otimes k}$ be the $k$-fold product measure of $\mu$.
For  $p\in \mathcal{P}(\mu)$,
we define $p^{\otimes k}\in \mathcal{P}(\mu^{\otimes k})$ by 
\[
p^{\otimes k}(\bm{x})\coloneqq
\prod_{m=1}^k p(\pi_m^k( \bm{x})) 
\quad \text{for $\bm{x}\in {\mathcal{X}}^k$},
\]
where $\pi_m^k\colon \mathcal{X}^k \to \mathcal{X}$ denotes the projection onto the $m$th coordinate.
Assume that  $p_\ast$ lies in  $\mathcal{M}\subset \mathcal{P}(\mu)$.
Then 
$p_\ast$  is estimated as a maximizer of the likelihood function given by
\[
p\mapsto e_{{\bm x_\ast}}(p^{\otimes k})
\quad \text{on }\mathcal{M}.
\]
In addition, if  $\mathcal{M}$ is an exponential family equipped with an exponential coordinate representation
$(\mathcal{M},\theta,T,c, \psi)$,
then $p_\ast$ becomes  a unique minimizer  of the dual problem 
\[
p\mapsto 
 \KL (\rho, p^{\otimes k}) \quad \text{on }\mathcal{M}
\]
for a suitable $\rho\in \mathcal{P}(\mu^{\otimes k})$,
where we used the fact that 
$\{p^{\otimes k} \mid p\in \mathcal{M}\}$ becomes an exponential family in $\mathcal{P}(\mu^{\otimes k})$
due to the expression 
\begin{align}\label{expindep}
p^{\otimes k}(\bm{x})
=\exp\left( \sum_{m=1}^k \left( \langle \theta(p), T(\pi_m^k( \bm{x}))\rangle-c(\pi_m^k( \bm{x})) \right) -k\psi(p) \right)
\end{align}
for $p\in \mathcal{M}$ and ${\bm x}\in \mathcal{X}^k$
(see Section~\ref{proof3} for more detail).
Moreover,
since the map defined~by
\[
p\mapsto \int_{\mathcal{X}} T(x)p(x)d\mu(x)\quad\text{on }\mathcal{M} 
\]
induces a global coordinate system on $\mathcal{M}$
(see for example \cite{BN1970}*{Theorem~5.5}),
$p_\ast \in \mathcal{M}$ is also deduced from an infinite independent trials by  the strong law of large numbers,
that is,  
\begin{align*}
\frac1k  \sum_{m=1}^k T(X_m)\xrightarrow{k\to \infty} \int_{\mathcal{X}}T(x)p(x)d\mu(x)
\quad\text{$\mathbb{P}$-a.s.}
\end{align*}
for $p\in \mathcal{M}$ and  i.i.d.\,random variables $(X_k)_{k\in\mathbb{N}}$ with distribution $p\mu$. 

Motivated by the fact that the law of i.i.d random variables with distribution in
an exponential family has the expression~\eqref{expindep}, 
we define the $(h,\tau)$-dependent 
 and  identically distributed random variables according to an element in an $(h,\tau)$-exponential family.
\begin{definition}\label{repeat}
Let $(h,\tau,I)\in \mathcal{G}$. 
An $(h,\tau)$-exponential family $\mathcal{M}$ is \emph{identically repeatable}
if, for each $k\in \mathbb{N}$,
there exist a positive and increasing smooth function $\alpha_k$ on $I$, a positive smooth function $b_k$ on $\mathcal{M}$ and an injection $\iota_k \colon \mathcal{M}\to \mathcal{P}_I(\mu^{\otimes k})$ 
such that $\iota_k(\mathcal{M})$ is a statistical model in $\mathcal{P}(\mu^{\otimes k})$ and,
for each global $(h,\tau)$-exponential coordinate representation $(\mathcal{M},\theta,T,c,\psi)$
on $\mathcal{M}$,
there exists $\psi_k\in C^{\infty}(\mathcal{M})$ such that  
\[
\iota_k(p)({\bm x})
=\alpha_k
\left(
\exp_{h,\tau}\left(  b_k(p) \sum_{m=1}^k \left( \langle \theta(p), T(\pi_m^k( \bm{x}))\rangle-c(\pi_m^k( \bm{x})) \right) -\psi_k(p) \right)
\right)
\]
for $p\in \mathcal{M}$ and  ${\bm x} \in {\mathcal{X}}^k$, and furthermore
\begin{align}\label{MC}
\int_{\mathcal{X}^{k+k'}}\iota_{k+k'}(p)({\bm x},{\bm y})d\mu^{\otimes k'}(\bm{ y})
=\iota_k(p)({\bm x})
\quad \text{for }k'\in \mathbb{N}.
\end{align}
We call the injection $(\iota_k)_{k\in \mathbb{N}}$ the \emph{repetition}.

For an identically repeatable $(h,\tau)$-exponential family $\mathcal{M}$ equipped with a repetition $(\iota_k)_{k\in\mathbb{N}}$
and $p\in \mathcal{M}$, 
we say random variables $(X_k)_{k\in \mathbb{N}}$ 
are \emph{$(h,\tau)$-dependent and identically distributed according to $p$}
if the law of $(X_m)_{m=1}^k$ is $\iota_k(p) \mu^{\otimes k}$ for each $k\in \mathbb{N}$.
\end{definition}

Since the global $(h,\tau)$-exponential normalization $\psi$ is determined by the other ingredients $(\mathcal{M},\theta, T, c)$,
we use  these ingredients to define the identical repeatability.
The function $\alpha_k$ is used to satisfy  the marginal condition~\eqref{MC},
where its monotonicity ensures that the relative order of event probabilities remains unchanged.
The function $b_k$ is a kind of  normalization associated to $\alpha_k$.
Note that any exponential family can be identically repeatable, where $\alpha_k=\mathrm{id}_I, b_k= \mathbf{1}_{\mathcal{M}}$ and $\psi_k=k\psi$
(see~\eqref{expindep}).

To establish a kind of the law of large number for  
$(h,\tau)$-dependent and identically distributed random variables, 
for $q>1$, define $h_q\colon (0,\infty)\to \mathbb{R}$~by
\[
h_q(r)\coloneqq\int_{1}^r \frac{t^{1-q}-1}{1-q}dt.
\]
We easily see $(h_q,\mathrm{id}_{(0,\infty)},(0,\infty))\in \mathcal{G}$.
\begin{theorem}~\label{mainthm3}
Let $d\in \mathbb{N}$ and $q>1$ with $d(q-1)<2$.
For $v\in \mathbb{R}^d$, define $p^v\colon\mathbb{R}^d\to \mathbb{R}$~by
\[
 p^v(x)\coloneqq\left[(q-1)|x-v|^2+ \left(\sqrt{\frac{q-1}{\pi}} \frac{\Gamma(\frac{1}{q-1})}{\Gamma(\frac{3-q}{2(q-1)})}\right)^{\frac{2(1-q)}{3-q}}\right]^{\frac{1}{1-q}}.
\]
Then the set $\mathcal{M}\coloneqq \{p^v\mid v\in \mathbb{R}^d \}$ becomes an identically repeatable $(h_q,\mathrm{id}_{(0,\infty)})$-exponential  family.
For its repetition $(\iota_k)_{k\in\mathbb{N}}$ and $p\in \mathcal{M}$, 
an $(h_q,\mathrm{id}_{(0,\infty)})$-dependent and identically distributed random variables $(X_k)_{k\in \mathbb{N}}$ according to $p$ satisfies 
\begin{align}\label{lln}
\frac1k  \sum_{m=1}^k T(X_m)\xrightarrow{k\to \infty} \int_{\mathbb{R}^d} T(x)\iota_1(p)(x)dx
\quad\text{$\mathbb{P}$-a.s.},
\end{align}
where $(\mathcal{M},\theta, T, c,\psi)$ is an 
$(h_q,\mathrm{id}_{(0,\infty)})$-exponential coordinate representation
on $\mathcal{M}$.
\end{theorem}
We discuss another identically repeatable $(h_q,\mathrm{id}_{(0,\infty)})$-exponential family 
 with mean-like and covariance-like parameters in Theorem~\ref{mainthm33}.

The rest of this paper is organized as follows. 
In Section \ref{pre}, 
we give the definition of a statistical model and 
review the relation between a Hessian structure and  a contrast function.
Section~\ref{Div} is devoted to the study of the $(h,\tau)$-divergence.
In Section~\ref{DEF}, we give the definition of an $(h,\tau)$-exponential family and develop its theory,
in particular, we establish the Pythagorean relation in Section~\ref{6}.
We prove Theorem~\ref{mainthm2} in Section~~\ref{pfthm2}
and Theorem~\ref{mainthm3} in Section~\ref{proof3}, respectively.

\section{Preliminaries}\label{pre}
In this section, we fix our notation and give the definition of a statistical model.
Then we recall the notion of a Hessian structure and discuss its relation to a contrast function.
\subsection{Notation}
Throughout this paper, a manifold is assumed to be smooth, Hausdorff, second countable, connected and without boundary.
We denote by $\langle \cdot, \cdot \rangle$ the standard inner product on $\mathbb{R}^d$. 
For a set~$E$, we denote by ${\bf1}_E$ and~$\mathrm{id}_E$ the indicator function of~$E$ and the identify map on $E$, respectively.
Let $n,N,d\in \mathbb{N}$.

Unless we indicate otherwise, we shall always use $\mathcal{M}$ to denote an $n$-dimensional manifold.
For a coordinate system $(U, \xi=(\xi^i)_{i=1}^{n})$  on $\mathcal{M}$ and  $1\leq i\leq n$,  
we denote by $\partial_i$ the $i$th coordinate vector and set $V^i\coloneqq V(\xi^i)$ for $V\in \mathfrak{X}(\mathcal{M})$ as usual.
We also fix  a topological space~${\mathcal{X}}$ 
and a Radon measure $\mu$ on ${\mathcal{X}}$ such that the support of $\mu$ is ${\mathcal{X}}$.
For  an open interval $I$ contained in $(0,\infty)$, set 
\begin{align*}
\mathcal{P}_I( \mu)
\coloneqq \{p\in L^1(\mu) \mid p\in I \text{ $\mu$-a.e. and } \|p\|_{L^1(\mu)}=1 \}
\quad
\text{and}\quad
\mathcal{P}( \mu)\coloneqq \mathcal{P}_{(0,\infty)}( \mu).
\end{align*}

\begin{definition}
For $x\in \mathcal{X}$, 
the \emph{evaluation function} $e_x\colon \mathcal{P(\mu)}\to (0,\infty)$ is defined by 
\[
e_x(p)\coloneqq p(x).
\]
\end{definition}
\begin{definition}
Let $I$ be an open interval contained in $(0,\infty)$ and  $\mathcal{P} \subset \mathcal{P}_I(\mu)$.
For a function $\Phi$ on $\mathcal{X}$ and a function $\phi$ on $I$, 
the expression $\mathbb{I}_{\Phi\phi} \in \mathbb{R}^{\mathcal{P}}$ means $\Phi \cdot (\phi \circ p)\in L^1(\mu)$ for any $p\in \mathcal{P}$.
In this case, we can define a function $\mathbb{I}_{\Phi\phi}\colon \mathcal{P}\to \mathbb{R}$ by 
\[
\mathbb{I}_{\Phi\phi}(p)\coloneqq \int_{\mathcal{X}} \Phi(x)\phi(p(x)) d\mu(x).
\]
We write
\[
\mathbb{I}_{\phi}\coloneqq \mathbb{I}_{\mathbf{1}_{\mathcal{X}} \phi},
\qquad
\mathbb{I}_{\Phi}\coloneqq \mathbb{I}_{\Phi \mathrm{id}_I},\
\qquad
\mathbb{I}\coloneqq \mathbb{I}_{\mathbf{1}_{\mathcal{X}} \mathrm{id}_I}.
\]
By abuse of notation, we use the same notation for a map $\Phi$ from $\mathcal{X}$ to $\mathbb{R}^N$,
where $\mathbb{I}_{\Phi\phi}$ is a map from $\mathcal{P}$ to $\mathbb{R}^N$.
\end{definition}
Note that  $\mathbb{I}\in \mathbb{R}^{\mathcal{P}(\mu)}$ and $\mathbb{I}\equiv 1$ on $\mathcal{P}(\mu)$.
\subsection{Statistical model}
In order to analyze a manifold in $\mathcal{P}(\mu)$ in the context of differential geometry,
we introduce a notion of statistical model. 
\begin{definition}\label{compati}
A manifold $\mathcal{M}$ in $\mathcal{P}(\mu)$ is called a \emph{statistical model} 
if the following five conditions hold.
\begin{enumerate}
\renewcommand{\theenumi}{C\arabic{enumi}}
\renewcommand{\labelenumi}{(\theenumi)}
 \setlength{\leftskip}{-9pt}
\item\label{C5}
There exists a $\mu$-null set $\mathcal{N}\subset \mathcal{X}$ such that 
the topology of $\mathcal{M}$ coincides with  the induced topology by the family of functions  $(e_x)_{x\in \mathcal{X}\setminus\mathcal{N}}$. 
\item\label{C1}
For $x \in \mathcal{X}\setminus\mathcal{N}$, 
the function $e_x\colon \mathcal{M}\to \mathbb{R}$ is smooth.
\item\label{C2}
For $p\in \mathcal{M}$, a function on ${\mathcal{X}}$ sending $x$ to 
the derivatives of $e_x$ of any order at~$p$ is Borel.
\item\label{C3}
For $V\in \mathfrak{X}(\mathcal{M})$ and $p\in \mathcal{M}$,
$V_p(e_x)=0$ holds for all $x\in {\mathcal{X}}\setminus\mathcal{N}$
if and only if $V_p=0$.
\item\label{C4}
For $V,W\in \mathfrak{X}(\mathcal{M}), p\in \mathcal{M}$,
the two functions $x\mapsto V_p (e_x)$, $x\mapsto V_p(W e_x)$ belong to $L^1(\mu)$ and
\begin{align*}
\int_{{\mathcal{X}}} V_p( e_x) d\mu(x)=0,
\qquad
\int_{{\mathcal{X}}} V_p( We_x) d\mu(x)=0.
\end{align*}
\end{enumerate}
\end{definition}
For example, the manifold of Gaussian densities on~$\mathbb{R}^d$ is a statistical model in $\mathcal{P}(\mathcal{L}^d)$,
where $\mathcal{L}^d$ denotes  the $d$-dimensional Lebesgue measure on $\mathbb{R}^d$.
Also if ${\mathcal{X}}$ is a finite set and $\mu$ is the counting measure on ${\mathcal{X}}$, 
then $\mathcal{P}(\mu)$ itself is a statistical model.
We will verify in Lemma~\ref{lem} that Conditions \eqref{C5}--\eqref{C3} are natural requirements for an $(h,\tau)$-exponential family.
Since $\mathbb{I}\equiv 1$ on $\mathcal{P}(\mu)$,
Condition \eqref{C4} means that the integration and the differentiation  of $e_x$ commute, that is,
\begin{align*}
V_p(\mathbb{I})=\int_{{\mathcal{X}}} V_p( e_x) d\mu(x),\qquad
V_p(W\mathbb{I})=\int_{{\mathcal{X}}} V_p(W e_x) d\mu(x).
\end{align*}
Note that, for a sequence $(p_k)_{k\in \mathbb{N}}$ in $\mathcal{M}$ and $p\in \mathcal{M}$,
$(p_k)_{k\in \mathbb{N}}$ converges to $p$ in $\mathcal{M}$ 
if and only if $(p_k)_{k\in \mathbb{N}}$ converges to $p$ $\mu$-almost everywhere.

\subsection{Hessian structure}
Let us review the notion of Hessian structure.
\begin{definition}\label{basicdef}
Let $g$ and $\nabla$ be a Riemannian metric and a connection on $\mathcal{M}$, respectively.
\begin{enumerate}
 \setlength{\leftskip}{-15pt}
\item
A coordinate system $(U,\xi)$ on $\mathcal{M}$ is \emph{affine} with respect to $\nabla$ 
if all the Christoffel symbols for $\xi$ vanish everywhere on~$U$.
\item
We say that $\nabla$ is \emph{flat} if its torsion and curvature vanish everywhere on $\mathcal{M}$.
\item
A pair $(g,\nabla)$ is called a \emph{Hessian structure} on $\mathcal{M}$ if $\nabla$ is flat and, for each point $p\in \mathcal{M}$, there exists a smooth function~$\Psi$ 
defined on an open neighborhood $U$ of $p$ in $\mathcal{M}$ such that $g=\nabla d \Psi$ on $U$.
We call $\Psi$ a \emph{local potential} on $U$ of $(g,\nabla)$.
\item
A pair $(\overline{g},\overline{\nabla})$ of a Riemannian metric and a connection on $\mathcal{M}$
is \emph{$1$-conformally equivalent} to $(g,\nabla)$ if 
there exists a positive smooth function $\zeta$ on $\mathcal{M}$ such that
\begin{align*}
\overline{g}=\zeta g,\qquad 
\overline{g}(\overline{\nabla}_V W,Z)=\overline{g}(\nabla_V W,Z)-Z(\log \zeta) \overline{g}(V,W)
\end{align*}
hold on $\mathcal{M}$ and for $V,W,Z\in \mathfrak{X}(\mathcal{M})$.
We call $\zeta$ a \emph{$1$-conformal factor}.
\end{enumerate}
\end{definition}
A condition for a connection to be flat is known.
\begin{proposition}{\rm(\cite{Shima}*{Proposition~1.1})}\label{straightforward}
A connection $\nabla$ on $\mathcal{M}$ is flat if and only if, for each $p\in \mathcal{M}$, there exists an affine coordinate system around $p$ with respect to~$\nabla$.
\end{proposition}
If a coordinate system $(U,\xi)$ is affine with respect to a connection $\nabla$, 
then
\begin{equation}\label{flatness}
(\nabla d\varphi)(\partial_i,\partial_j)=\partial_i\partial_j \varphi
\quad\text{for }
\varphi \in C^\infty(U)
\text{ and }1\leq i,j\leq n.
\end{equation}
%
\subsection{Contrast function}
We briefly review the properties of a contrast function.
For $V\in \mathfrak{X}(\mathcal{M})$, 
we set 
$
\widetilde{V}\coloneqq (V,0), \widehat{V}\coloneqq (0,V)\in \mathfrak{X}(\mathcal{M}\times \mathcal{M}).
$
We have
\[
\widetilde{V}\widetilde{W}=\widetilde{VW},\qquad \widehat{V}\widehat{W}=\widehat{VW},\qquad
\widetilde{V}\widehat{W}=\widehat{W}\widetilde{V},
\]
which are generally differ from $(V,W)\in \mathfrak{X}(\mathcal{M}\times \mathcal{M})$.

Let $D$ be a contrast function on $\mathcal{M}$.
For a curve  $\varepsilon \mapsto \sigma_{V}^{\varepsilon}(p)$ 
passing through $p\in\mathcal{M}$ at $\varepsilon=0$ with initial velocity~$V_p$,
if $V_p \neq 0$ and $\varepsilon$ is small enough, then 
the function $\varepsilon \mapsto D(p, \sigma_{V}^{\varepsilon}(p))$
is nonnegative and equal to zero if and only if $\varepsilon=0$, consequently
\begin{align*}
\frac{d}{d\varepsilon} D(p, \sigma_{V}^{\varepsilon}(p)) \bigg|_{\varepsilon=0}=0, \qquad
\frac{d^2}{d\varepsilon^2} D(p, \sigma_{V}^{\varepsilon}(p)) \bigg|_{\varepsilon=0}
\geq 0.
\end{align*}
We will freely use these properties in the sequel.

Apart from Condition~\ref{Eguchi},
a requirement of contrast functions in \cite{AA}*{Eq.(2)} is interpreted as follows.
\begin{condition}\label{AyA}
For $V\in \mathfrak{X}(\mathcal{M})$ and $p\in \mathcal{M}$ with $V_p\neq0$, 
let $\varepsilon \mapsto \sigma_{V}^{\varepsilon}(p)$ be a curve 
passing through $p$ at $\varepsilon=0$ with initial velocity~$V_p$.
Then
\[
\frac{d^2}{d\varepsilon^2} D(p, \sigma_{V}^{\varepsilon}(p)) \bigg|_{\varepsilon=0}> 0.
\]
\end{condition}
%
We show the equivalence between Condition~\ref{Eguchi} and Condition~\ref{AyA}.
This equivalence is claimed in~\cite{AA}*{Eq.(2)} without proof.
\begin{proposition}\label{equiv}
For a contrast function $D$ on $\mathcal{M}$, 
\[
-\widetilde{V}_{(p,p)}(\widehat{V}D)=\frac{d^2}{d\varepsilon^2} D(p,\sigma_{V}^{\varepsilon}(p)) \bigg|_{\varepsilon=0}
\qquad\text{for $V\in \mathfrak{X}(M)$ and $p\in M$,}
\]
where $\varepsilon \mapsto \sigma_{V}^{\varepsilon}(p)$ is a curve 
passing through $p$ at $\varepsilon=0$ with initial velocity~$V_p$.
Consequently 
Condition~{\rm\ref{Eguchi}} and Condition~{\rm\ref{AyA}} are equivalent to each other.
\end{proposition}
\begin{proof}
We first show that 
the second derivative of the function $\varepsilon \mapsto D(p,\sigma_{V}^{\varepsilon}(p))$ at $\varepsilon =0$ is independent of the choice of curves~$\varepsilon \mapsto \sigma_{V}^{\varepsilon}(p)$.
Let $(U, \xi)$ be a coordinate system around $p$.
It turns out that 
\begin{align*}
\partial_i D(p,\cdot)\big|_{p}=
\frac{d}{d\varepsilon} D(p,\sigma_{\partial_i}^{\varepsilon}(p))\bigg|_{\varepsilon=0}
=0.
\end{align*}
Then a direct calculation gives 
\begin{align*}
\frac{d^2}{d\varepsilon^2} D(p,\sigma_{V}^{\varepsilon}(p))\bigg|_{\varepsilon=0}
&=\sum_{i,j=1}^{n}
\left( V^j_p \cdot \partial_j V^i|_p \cdot \partial_i D(p,\cdot)\big|_p +V^i_p V^j_p \cdot\partial_i \partial_j D(p,\cdot)\big|_p \right)\\
&=\sum_{i,j=1}^{n} V_p^iV_p^j \cdot \partial_i \partial_j D(p,\cdot)\big|_p,
\end{align*}
which depends only on $V$ and $p$, not on the choice of curves $\varepsilon \mapsto \sigma_{V}^{\varepsilon}(p)$.

Next we show 
\[
-\widetilde{W}_{(p,p)}(\widehat{V}D)=\widehat{V}_{(p,p)}(\widehat{W}D).
\]
Although this  is proved  in~\cite{Eguchi92}*{p.632, 633}, we prove it for the sake of completeness.
Define a map $\mathrm{diag}\colon \mathcal{M} \to \mathcal{M}\times \mathcal{M}$ by $\mathrm{diag}(p)\coloneqq (p,p)$.
Then it turns out that 
\begin{align*}
\left((\widehat{V}D) \circ \mathrm{diag}\right)(p)
&=\widehat{V}_{(p,p)}D
=\frac{d}{d\varepsilon}D(p, \sigma_{V}^{\varepsilon}(p)) \bigg|_{\varepsilon=0}
=0
\end{align*}
hence $W( (\widehat{V}D) \circ \mathrm{diag})=0$ on $\mathcal{M}$.
A direct computation provides
\begin{align*}
0=W_p\left(( \widehat{V}D) \circ \mathrm{diag} \right)
&=\frac{\partial}{\partial\varepsilon'}\left(\frac{\partial}{\partial\varepsilon} D(\sigma_{W}^{\varepsilon'}(p),\sigma_{V}^\varepsilon(\sigma_{W}^{\varepsilon'}(p))) \right)\bigg|_{\varepsilon=0}\bigg|_{\varepsilon'=0}\\
&=
\frac{\partial^2}{\partial\varepsilon'\partial\varepsilon} D(p,\sigma_{V}^\varepsilon(\sigma_{W}^{\varepsilon'}(p))) \bigg|_{\varepsilon=0,\varepsilon'=0}
+
\frac{\partial^2}{\partial\varepsilon'\partial\varepsilon} D(\sigma_{W}^{\varepsilon'}(p),\sigma_{V}^{\varepsilon}(p)) \bigg|_{\varepsilon=0,\varepsilon'=0}\\
&
=\widehat{W}_{(p,p)}(\widehat{V}D)+\widetilde{W}_{(p,p)}(\widehat{V}D),
\end{align*}
implying $-\widetilde{W}_{(p,p)}(\widehat{V}D)=\widehat{W}_{(p,p)}(\widehat{V}D)$.
In addition, we find 
\[
\widehat{V}_{(p,p)}(\widehat{W}D)-\widehat{W}_{(p,p)}(\widehat{V}D)
=\widehat{[V,W]}_{(p,p)}(D)
=\left(( \widehat{[V,W]}D) \circ \mathrm{diag} \right)(p)
=0
\]
hence $\widehat{W}_{(p,p)}(\widehat{V}D)=\widehat{V}_{(p,p)}(\widehat{W}D)$.
These show $-\widetilde{W}_{(p,p)}(\widehat{V}D)=\widehat{V}_{(p,p)}(\widehat{W}D)$.

If $\varepsilon \mapsto \sigma_V^\varepsilon(p)$ is the integral curve of $V$ passing through~$p$ at $\varepsilon=0$,
then, by definition,
\begin{align*}
\widehat{V}_{(p, \sigma^\varepsilon_V(p))}D&=
\frac{d}{d\varepsilon} D(p,\sigma_{V}^{\varepsilon}(p)),\\
-\widetilde{V}_{(p,p)}(\widehat{V}D)&=\widehat{V}_{(p, p) } ( \widehat{V} D)
=
V_p( \widehat{V}_{(p,\cdot) } D)
=\frac{d}{d\varepsilon} \widehat{V}_{(p,\sigma_{V}^{\varepsilon}(p)) }D \bigg|_{\varepsilon=0}
=
\frac{d^2}{d\varepsilon^2} D(p,\sigma_{V}^{\varepsilon}(p)) \bigg|_{\varepsilon=0}
\end{align*}
follows.
This completes the proof.
\end{proof}
We recall two known facts of geometric structure induced by contrast functions.
Since the proof is straightforward and found in the reference, we omit it. 
\begin{proposition}{\rm(\cite{Kurose}*{Remark in p.431})}
Let $D$ be a contrast function on~$\mathcal{M}$ satisfying Condition~\rm{\ref{Eguchi}}.
For a positive smooth function $\zeta$ on $\mathcal{M}$, define 
\[
\overline{D}(p,p')\coloneqq \zeta(p) D(p,p')\quad \text{for }p,p' \in \mathcal{M}.
\] 
Then $\overline{D}$ is a contrast function on $\mathcal{M}$ satisfying Condition~\rm{\ref{Eguchi}}.
Define $(g,\nabla)$ (resp.\,$(\overline{g},\overline{\nabla}$) by~\eqref{met} associated to~$D$ (resp.\,$\overline{D})$.
Then
$(\overline{g},\overline{\nabla})$ 
is 
$1$-conformally equivalent to 
$(g,\nabla)$,
where  $\zeta$ becomes a $1$-conformal factor.
\end{proposition}
\begin{proposition}\label{cano}{\rm(cf.\cite{Shima}*{Section~2.4})}
Let $(g,\nabla)$ be a Hessian structure on $\mathcal{M}$.
For an affine coordinate system $(U, \xi)$ with respect to $\nabla$ and the local potential $\Psi$ on~$U$ of $(g,\nabla)$, 
define $\eta\colon U\to\mathbb{R}^n$
and $\overline{D}\colon U\times U \to \mathbb{R}$ by
\begin{align*}
\eta\coloneqq \left(\partial_i \Psi\right)_{i=1}^{n},\qquad
\overline{D}(p,p')\coloneqq \Psi(p')-\Psi(p)+\langle \xi(p)-\xi(p'),\eta(p)\rangle.
\end{align*}
Then $\overline{D}$ is a contrast function on $U$ satisfying Condition~\rm{\ref{Eguchi}},
which is  called the \emph{canonical divergence} of the Hessian structure $(g,\nabla)$ on $U$.
The Riemannian metric and the connection defined by \eqref{met} associated to $\overline{D}$ is the restriction of $g$ and~$\nabla$ to~$U$, respectively.
\end{proposition}
In Proposition~\ref{cano}, $(U,\eta)$ becomes a coordinate system by the nonsingularity of 
\[
\left(\frac{\partial}{\partial \xi^j} \eta_i\right)_{1\leq i,j\leq n}
=
\left(\frac{\partial^2}{\partial \xi^i\partial \xi^j} \Psi\right)_{1\leq i,j\leq n}
=\left( g\left(\frac{\partial}{\partial \xi^i},\frac{\partial}{\partial \xi^j} \right)\right)_{1\leq i,j\leq n}.
\]
The canonical divergence is defined only locally, not globally on a manifold in general.

\section{$(h,\tau)$-divergence}\label{Div}
We discuss  the $(h,\tau)$-divergence and its equivalence relation.
\begin{definition}\label{gauge}
\begin{enumerate}
\item
We set 
\begin{align*}
\mathcal{T}&\coloneqq \left\{(\tau,I) \bigm| \text{$I\subset (0,\infty)$ is an open interval and $\tau \in C^{\infty}(I)$ with $\tau'>0$ on $I$} \right\},\\
\mathcal{G}
&\coloneqq \left\{(h,\tau,I)\bigm| \text{$(\tau,I) \in \mathcal{T}$ and $h \in C^\infty(\tau(I))$ with $h''>0$ on $\tau(I)$} \right\}.
\end{align*}
\item
 \setlength{\leftskip}{-15pt}
For $(h, \tau,I)\in \mathcal{G}$, define $d_{h,\tau}\colon I\times I \to \mathbb{R}$ by 
\begin{align*}
d_{h,\tau}(t,s)\coloneqq h(\tau(t))-h(\tau(s))- \left(\tau(t)-\tau(s)\right) h'(\tau(s)).
\end{align*}
We also define the three smooth functions 
$\ell_{h,\tau}, m_{h,\tau},\gamma_{h,\tau}$ on $I$ respectively by
\begin{align*}
\ell_{h,\tau}\coloneqq h'\circ \tau, \qquad
m_{h,\tau}(t)\coloneqq -\frac{\partial^2}{\partial t\partial s}d_{h,\tau}\bigg|_{(t,t)},\qquad
\gamma_{h,\tau}(t)\coloneqq -\frac{\partial^3}{\partial t \partial s^2}d_{h,\tau}\bigg|_{(t,t)}.
\end{align*}
\item
On $\mathcal{G}$, define an equivalence relation $(h_1, \tau_1,I_1) \simeq (h_2, \tau_2,I_2)$ to mean that $I_1=I_2$ and $d_{h_1,\tau_1}=d_{h_2,\tau_2}$.
\end{enumerate}
\end{definition}
For $(h,\tau, I)\in \mathcal{G}$, 
we see that $(\ell_{h,\tau}, I)\in \mathcal{T}$,  $d_{h,\tau}\geq 0$  and 
\begin{align}\label{corresponding}
d_{h,\tau}(t,s)=\int_{s}^t (\ell_{h,\tau}(u)-\ell_{h,\tau}(s)) \tau'(u)du.
\end{align}
It follows from 
\[
\ell'_{h,\tau}=(h'' \circ \tau)\cdot \tau'>0
\]
that  $d_{h,\tau}(t,s)\in[0,\infty)$ and $d_{h,\tau}(t,s)=0$ holds if and only if $t=s$.
Thus~$d_{h,\tau}$ is a contrast function on $I$.
Moreover, $d_{h,\tau}$ satisfies Condition~\ref{Eguchi} since
\[
-\frac{\partial^2}{\partial t\partial s}d_{h,\tau} \bigg|_{(t,t)}=m_{h,\tau}(t)=
 \ell_{h,\tau}'(t)\tau'(t)
=h''(\tau(t))\tau'(t)^2>0.
\]
Note that $\gamma_{h,\tau} =\ell''_{h,\tau}\cdot \tau' $.

By the nonnegativity of $d_{h,\tau}$, we can define  a divergence on $\mathcal{P}_{I}(\mu)$ by 
\[
D_{h,\tau}(p,p')\coloneqq \int_{{\mathcal{X}}} d_{h,\tau}(p(x), p'(x)) d\mu(x)
\]
and we call it the \emph{$(h,\tau)$-divergence}.

The equivalence class of $(h,\tau,I)$ is completely determined by $(m_{h,\tau},\gamma_{h,\tau},I)$.
\begin{lemma}\label{equivclass}
For $(h_1, \tau_1,I),(h_2, \tau_2,I)\in \mathcal{G}$, the following three conditions are equivalent to each others.
\begin{enumerate}
\renewcommand{\theenumi}{\roman{enumi}}
\renewcommand{\labelenumi}{(\theenumi)}
 \setlength{\leftskip}{-9pt}
\item\label{D1}
$(h_1, \tau_1,I)\simeq (h_2, \tau_2,I)$.
\item\label{D2}
There exist $a_1, a_2,a_3 \in \mathbb{R}$ and $\lambda>0$ such that
\begin{equation*}
\qquad 
h_{1}(r)=h_{2}(\lambda r+a_3)+a_1r+a_2 \quad\text{for } r\in \tau_1(I),\qquad
\tau_2(t)=\lambda \tau_1(t)+a_3 \quad\text{for } t\in I.
\end{equation*}
\item\label{D3}
$(m_{h_1,\tau_1},\gamma_{h_1,\tau_1})=(m_{h_2,\tau_2},\gamma_{h_2,\tau_2})$.
\end{enumerate}
\end{lemma}
\begin{proof}
Since $m_{h,\tau}$ and $\gamma_{h,\tau}$ are  the partial derivatives of $d_{h,\tau}$ for $(h,\tau, I)\in \mathcal{G}$, 
the implication from \eqref{D1} to \eqref{D3} is trivial.
Assume \eqref{D2}.
Then a direct calculation gives
\begin{align*}
\ell_{h_1,\tau_1}(t)&=h_{1}'(\tau_1(t))=\lambda h_{2}'(\lambda \tau_1(t)+a_3)+a_1=\lambda \ell_{h_2,\tau_2}(t)+a_1,\\
d_{h_1,\tau_1}(t,s)
&=\int_t^s \left(\ell_{h_1,\tau_1}(u)-\ell_{h_1,\tau_1}(s)\right)\tau_1'(u)du\\
&=\int_t^s \lambda \left(\ell_{h_2,\tau_2}(u)-\ell_{h_2,\tau_2}(s)\right)\lambda^{-1}\tau_2'(u)du\\
&=d_{h_2,\tau_2}(t,s),
\end{align*}
which yields \eqref{D1}.

The rest is to prove the implication from \eqref{D3} to \eqref{D2}.
Assume \eqref{D3}.
Then we have
\[
\frac{\ell''_{h_1,\tau_1}}{\ell'_{h_1,\tau_1}}
=\frac{\ell''_{h_1,\tau_1}\cdot \tau'_1}{\ell'_{h_1,\tau_1}\cdot \tau'_1}
=\frac{\gamma_{h_1,\tau_1}}{m_{h_1,\tau_1}}
=\frac{\gamma_{h_2,\tau_2}}{m_{h_2,\tau_2}}
=\frac{\ell''_{h_2,\tau_2}\cdot \tau'_2}{\ell'_{h_2,\tau_2}\cdot \tau'_2}
=\frac{\ell''_{h_2,\tau_2}}{\ell'_{h_2,\tau_2}}
\]
which implies that  there exists $\lambda>0$ such that 
\begin{equation}\label{lambda}
\ell'_{h_{1},\tau_1}=\lambda \ell'_{h_{2},\tau_2}\quad
\text{on }I.
\end{equation}
Combining this with 
\[
\ell'_{h_1,\tau_1}\cdot \tau'_{1}
=m_{h_1,\tau_1}
=m_{h_2,\tau_2}
=\ell'_{h_2,\tau_2}\cdot \tau'_{2}
\]
leads to $\tau'_{2}=\lambda \tau'_{1}$ on $I$,
that is, 
there exists $a_3 \in \mathbb{R}$ such that
\begin{equation}\label{a3}
\tau_2=\lambda \tau_1 +a_3 \quad\text{on } I.
\end{equation}
By $\ell'_{h_{j},\tau_j}=(h_j'' \circ \tau_j)\cdot \tau'_j $ on $I$, we observe from \eqref{lambda} and \eqref{a3} that 
\[
h_1''(r )=\lambda^2 h_2''(\lambda r+a_3) \quad \text{for } r\in \tau_1(I).
\]
Integrating this equality twice ensures the existence of $a_1,a_2 \in \mathbb{R}$ such that 
\[
h_{1}(r)=h_{2}(\lambda r+a_3)+a_1r+a_2 \quad\text{for } r\in \tau_1(I).
\]
This with \eqref{a3} implies \eqref{D2}.
\end{proof}
Since $m_{h,\tau}, \gamma_{h,\tau}$ determine  an equivalence relation on $\mathcal{G}$  by Lemma~\ref{equivclass}
and these two functions are the products of the derivatives of $\tau$ and $\ell_{h,\tau}$, 
we can consider an equivalence relation on $\mathcal{G}$
 by the use of $(\tau, \ell_{h,\tau},I)$.  
\begin{definition}\label{gauge2}
\begin{enumerate}
\item
We set
\[
\mathcal{T}^{(2)}\coloneqq \left\{(\tau,\ell,I) \mid (\tau,I),(\ell,I)\in \mathcal{T} \right\}.
\]
 \setlength{\leftskip}{-15pt}
\item
For 
$(\tau, \ell, I)\in \mathcal{T}^{(2)}$ and $a\in I$, 
we define $\delta_{\tau, \ell}\colon I\times I \to \mathbb{R}$ and $h_{\tau,\ell;a}\in C^\infty(\tau(I))$ 
respectively by 
\begin{align*}
\delta_{\tau, \ell}(t,s)
\coloneqq \int_{s}^t \left(\ell(u)-\ell(s)\right)\tau'(u) du,\qquad
h_{\tau,\ell;a}(r)\coloneqq 
\int_{a}^{\tau^{-1}(r)} \ell(t)\tau'(t) dt.
\end{align*}
\item
On $\mathcal{T}^{(2)}$, define an equivalence relation $ (\tau_1, \ell_1,I_1) \approx (\tau_2, \ell_2, I_2) $ to mean that 
$I_1=I_2$ and $\delta_{\tau_1, \ell_1}=\delta_{\tau_2, \ell_2}$.
\end{enumerate}
\end{definition}
\begin{lemma}\label{bij}
\begin{enumerate}
\setlength{\leftmargin}{23pt}
\setlength{\leftskip}{-15pt}
\item\label{1}
For $(\tau, \ell,I)\in \mathcal{T}^{(2)}$ and $a\in I$, it follows that $(h_{\tau,\ell;a}, \tau,I)\in \mathcal{G}$ and 
$\delta_{\tau,\ell}=d_{h_{\tau,\ell;a}, \tau}$ on $I\times I$ hence $(h_{\tau,\ell;a'}, \tau,I) \simeq (h_{\tau,\ell;a}, \tau,I)$ for any $a'\in I$.
\setlength{\leftskip}{-15pt}
\item\label{2}
A map from $ \mathcal{G}$ to $\mathcal{T}^{(2)}$ sending $(h,\tau, I)$ to $(\tau,\ell_{h,\tau}, I)$ induces a bijection from $ \mathcal{G}/\simeq$ to $\mathcal{T}^{(2)}/\approx$, 
where $d_{h,\tau}=\delta_{\tau, \ell_{h,\tau}}$.
\end{enumerate}
\end{lemma}
\begin{proof}
We first prove Claim~\eqref{1}.
Let $(\tau, \ell,I)\in \mathcal{T}^{(2)}$ and $a\in I$.
We calculate
\[
h''_{\tau,\ell;a}= (\ell \circ \tau^{-1})'= \left(\ell' \circ  \tau^{-1}\right)\cdot  \frac{d}{dr}\tau^{-1} >0,
\]
implying $(h_{\tau,\ell;a}, \tau,I)\in \mathcal{G}$.
We see that 
$
\ell_{h_{\tau,\ell;a},\tau}=h'_{\tau,\ell;a}\circ \tau=\ell,
$
and conclude from \eqref{corresponding} that
$d_{h_{\tau,\ell;a},\tau}=\delta_{\tau,\ell}$.
This yields $(h_{\tau,\ell;a'}, \tau,I)\simeq (h_{\tau,\ell;a}, \tau,I)$ for any $a'\in I$.
Thus Claim~\eqref{1} follows.

Let us prove Claim~\eqref{2}.
For $(h,\tau, I)\in \mathcal{G}$, 
we find $\ell_{h,\tau}\in C^{\infty}(I)$ and $\ell_{h,\tau}' >0$,
implying $(\tau,\ell_{h,\tau},I)\in \mathcal{T}^{(2)}$.
Moreover $d_{h,\tau}=\delta_{\tau,\ell_{h,\tau}}$ follows from~\eqref{corresponding}. 
By Claim~\eqref{1}, a map from $\mathcal{T}^{(2)}/\approx$ to $\mathcal{G}/\simeq$ sending 
the equivalence class of $(\tau,\ell, I)$ to that of $(h_{\tau,\ell;a}, \tau, I)$ 
is well-defined.
Thus  a map from  $\mathcal{G}/\simeq$  to $\mathcal{T}^{(2)}/\approx$ sending
the equivalence class of $(h,\tau, I)$ to that of $(\tau,\ell_{h,\tau}, I)$  is well-defined
and  its inverse map is given by a map sending the equivalence class of $(\tau,\ell, I)$ to that of $(h_{\tau,\ell;a}, \tau, I)$.
This proves Claim~\eqref{2}.
\end{proof}

There is an involution on $\mathcal{T}^{(2)}$ sending $(\tau,\ell, I)$ to $(\ell, \tau, I)$, 
where  $\delta_{\tau, \ell}(t,s)=\delta_{\ell, \tau}(s,t)$ holds.
This induces an involution on $\mathcal{G}/\simeq$, not $\mathcal{G}$ itself,  sending $(h,\tau,I)$ to $(h_{\ell_{h,\tau}, \tau;a},\ell_{h,\tau}, I)$.
If we consider a suitable subset of $\mathcal{G}$, the involution can be expressed as the conjugate.
This is roughly claimed in~\cite{Zhang}*{p.178}, \cite{NZ}*{p.88} and we  prove it rigorously.
For the sake, 
we briefly review some notions and results from convex analysis.
See~\cite{Rock} for more details.
\begin{definition}
Let $h\colon \mathbb{R}\to (-\infty, \infty]$.
\begin{enumerate}
 \setlength{\leftskip}{-15pt}
\item
The \emph{conjugate} of $h$  is the function $h^\ast\colon \mathbb{R}\to[-\infty,\infty]$ defined by 
\[
h^\ast(r^\ast)\coloneqq \sup\{ r^\ast r-h(r)\ |\ r\in \mathbb{R} \}.
\] 
\item
We say that $h$ is a \emph{convex function of Legendre type} if the following conditions hold.
\begin{itemize}
 \setlength{\leftskip}{-20pt}
\item
$h$ is  convex and  lower semicontinuous  on $\mathbb{R}$,
where by convention 
\begin{align*}
\infty\leq \infty,\qquad
a\cdot\infty=\infty,\qquad
b+\infty=\infty, \qquad
\text{for }a\in (0,\infty]
\text{ and  }b\in \mathbb{R}.
\end{align*}
\item
$h^{-1}(\mathbb{R}) \neq \emptyset$.
We denote by $J_h$ the interior of $h^{-1}(\mathbb{R})$.
\item 
$h$ is strictly convex on $J_h$.
\item
$h$ is differentiable in $J_h$ in addition to
\[
\qquad
\lim_{r \downarrow \inf J_h} h'(r)=-\infty
\quad \text{if $\inf J_h>-\infty$\quad and}\quad
\lim_{r \uparrow \sup J_h} h'(r)=\infty
\quad \text{if $\sup J_h<\infty$.}
\]
\end{itemize}
\end{enumerate}
\end{definition}
\begin{proposition}{\rm (\cite{Rock}*{Theorems~12.2, 26.4, 26.5})}\label{LC}
Let $h\colon \mathbb{R}\to (-\infty, \infty]$ be a convex function of Legendre type.
Then~$h^\ast$ is a convex function of Legendre type and $(h^\ast)^\ast=h$.
Moreover, 
$h'(J_h)=J_{h^\ast}$ and 
\[
h^\ast(h'(r))=rh'(r)-h(r)\quad \text{for }r\in J_{h}.
\]
\end{proposition}
\begin{proposition}\label{mainthm1}{\rm(cf.\cite{Zhang}*{p.178}, \cite{NZ}*{p.88})}
Set 
\begin{align*}
\mathcal{G}_{\ast}& \coloneqq\left\{ (h|_{\tau(I)},\tau, I) \in \mathcal{G}\Bigm|
\text{$h$ is a convex function of Legendre type with  $J_h=\tau(I)$}
\right\},\\
\mathcal{T}^{(2)}_{\ast} &\coloneqq\left\{ (\tau, \ell,I)\in \mathcal{T}^{(2)} \Biggm|
\begin{tabular}{l}
$\displaystyle
\lim_{t \downarrow \inf I} \ell(t)=-\infty 
 \text{ if }\inf \tau(I)>-\infty$\text{ and} \\
$\displaystyle \lim_{t \uparrow \sup I} \ell(t)=\infty
\text{  if } \sup \tau(I)<\infty$
\end{tabular}
\right\}.
\end{align*}
There exists a bijection from $ \mathcal{G}_\ast/\simeq$ to $\mathcal{T}^{(2)}_\ast/\approx$
such that an involution  on $\mathcal{G}_\ast/\simeq$ sending  the equivalence class of $(h_{\tau(I)},\tau, I)$ to that of $(h^\ast|_{\ell(I)}, \ell, I)$ 
corresponds to that  on $\mathcal{T}^{(2)}/\approx$ sending  the equivalence class of $(\tau,\ell, I)$ to that of $(\ell, \tau, I)$.
\end{proposition}
\begin{proof}
For  $(\tau, \ell,I)\in \mathcal{T}^{(2)}_\ast$ and  $a\in I$,
 define  $h\colon\mathbb{R}\to [-\infty, \infty]$  by
\begin{align*}
h(r)&\coloneqq 
\begin{cases}
h_{\tau,\ell;a} (r)&\text{if }r\in \tau(I),\\
\displaystyle
\liminf_{r'\in I, r'\to r} h_{\tau,\ell;a}(r) &\text{if }r\in \partial\tau(I),\\
\infty &\text{otherwise}.
\end{cases}
\end{align*}
It is trivial that $h^{-1}(\mathbb{R})\neq \emptyset$  with $J_h=\tau(I)$ and $h\in C^\infty(\tau(I))$.
By definition, we easily see the lower continuity of $h$  on $\mathbb{R}$.
Since we have $h''=h_{\tau,\ell;a}''>0$ on $J_h$,
$h$ is strictly  convex on $J_h$ and is convex  on $\mathbb{R}$.
If  $\inf J_h=\inf \tau(I)>-\infty$,
then it follows  from the assumptions that
\[
\lim_{r \downarrow \inf J_h} h'(r)
=
\lim_{t \downarrow \inf I} h'(\tau(t))
=
\lim_{t \downarrow \inf I} \ell(t)
=
-\infty
\]
hence $h(\inf \tau(I))>-\infty$.
Similarly, if $\sup J_{h}=\sup \tau(I)<\infty$, then 
we have
\[
\lim_{r \uparrow \sup J_h} h'(r)=\infty  \quad \text{and}\quad
h(\sup \tau(I))>-\infty.
\]
These show $h(\mathbb{R})\subset (-\infty, \infty]$.
Thus   $h$ is a convex function of Legendre type and $(h_{\tau(I)},\tau, I)\in \mathcal{G}_\ast$.
It follows from Proposition~\ref{LC} that
\[
J_{h^\ast}=h'(J_h)=(\ell\circ \tau^{-1})(\tau(I))=\ell(I).
\]
Moreover, for $r^\ast \in \ell(I)$, if we set $r\coloneqq \tau(\ell^{-1}(r^\ast))$, 
then we have  $h'(r)=r^\ast$ and again observe from Proposition~\ref{LC} that 
\begin{align*}
h^\ast(r^\ast)
&=h^\ast(h'(r))
=r h'(r)-h(r)\\
&=r\ell(\tau^{-1}(r))-\int_a^{\tau^{-1}(r)} \ell(t)\tau'(t) dt
=\ell(a)\tau(a)+  \int_a^{\ell^{-1}(r^\ast)} \ell'(t)\tau(t) dt\\
&=\ell(a)\tau(a)+h_{\ell, \tau;a}(r^\ast),
\end{align*}
which yields $(h^\ast|_{\ell(I)}, \ell, I)\simeq (h_{\ell, \tau;a},\ell, I)$ hence $(h^\ast|_{\ell(I)}, \ell, I)\in \mathcal{G}_\ast$.
These with Lemma~\ref{bij}\eqref{2}   show that 
a map from $ \mathcal{G}_\ast$ to $\mathcal{T}^{(2)}_\ast$ sending $(h|_{\tau(I)},\tau, I)$ to $(\tau,\ell_{h|_{\tau(I)},\tau}, I)$ induces a desired bijection from $ \mathcal{G}/\simeq$ to $\mathcal{T}^{(2)}/\approx$.
\end{proof}
%

\begin{remark}\label{underinv}
For $ (h|_{\tau(I)},\tau, I)\in \mathcal{G}_{\ast}$, 
we have 
\begin{align*}
m_{h|_{\tau(I)},\tau}=\ell'\cdot \tau'=m_{h^\ast|_{\ell(I)},\ell}, \qquad
\gamma_{h|_{\tau(I)},\tau}=\ell''\cdot \tau', \qquad
\gamma_{h^\ast|_{\ell(I)},\ell}=\tau''\cdot \ell'.
\end{align*}
This implies that 
the $(h|_{\tau(I)},\tau)$-divergence and  the $(h^\ast|_{\ell(I)},\ell)$-divergence induce the same Riemannian metric but they generally induce the different connections on $I$  
via \eqref{met}.
These two connections are said to be \emph{dual} to each other with respect to this  Riemannian metric 
(see for example~\cite{Shima}*{Definition~2.8}).
\end{remark}
\section{Deformed Exponential Family }\label{DEF}
An exponential family, especially a manifold of  Gaussian densities, together with the Kullback--Leibler divergence lays the foundation of information geometry.
We develop this foundation along an $(h,\tau)$-exponential family together with the $(h,\tau)$-divergence.

%
We briefly review the definition of an $(h,\tau)$-exponential family and its properties.
%
\begin{definition}
\label{exp}
Let $(h,\tau,I)\in \mathcal{G}$  and $\mathcal{M}$ be a set in $\mathcal{P}_I(\mu)$.
\begin{enumerate}
 \setlength{\leftskip}{-15pt}
\item
Define the smooth function $\chi_{h,\tau}\colon I \to \mathbb{R}$ by
\[
\chi_{h,\tau}\colon=\frac{1}{\ell'_{h,\tau}}.
\]
We extend the inverse function of $\ell_{h,\tau} \colon I \to \ell_{h,\tau}(I)$ to $\mathbb{R}$ by defining
\begin{align*}
\exp_{h,\tau}(u)\coloneqq \begin{cases}
\infty & \text{if } u\geq \sup \ell_{h,\tau}(I),\\
\ell_{h,\tau} ^{-1}(u) & \text{if } u\in \ell_{h,\tau}(I),\\
0 & \text{if }u \leq \inf \ell_{h,\tau}(I),
\end{cases}
\end{align*}
and call $\exp_{h,\tau}$ the \emph{$(h,\tau)$-exponential function}.
\item
An \emph{$(h,\tau)$-exponential representation of order $N$} on $\mathcal{M}$ is a quintuple $(U,\theta, T, c,\psi)$, consisting of
an open set $U$ in $\mathcal{M}$,
a map $\theta=(\theta^i)_{i=1}^N\colon U \to \mathbb{R}^N$,
a Borel map $T=(T_i)_{i=1}^N\colon \mathcal{X} \to\mathbb{R}^N$, 
a Borel function $c\colon \mathcal{X}\to \mathbb{R}$
and a function $\psi\colon U\to \mathbb{R}$ such that, 
for each $p\in U$, 
\begin{align}\label{exppp}
p(x)=\exp_{h,\tau}\left( \langle \theta(p), T(x) \rangle -c(x)-\psi(p) \right)
\quad\text{for $\mu$-a.e.\,${x} \in {\mathcal{X}}$.}
\end{align}
The function $\psi$ is called a local \emph{$(h,\tau)$-exponential normalization} on $U$.

In the case that $\mathcal{M}$ is a manifold, 
the quintuple $(U,\theta, T, c,\psi)$ is called an \emph{$(h,\tau)$-exponential coordinate representation} on $\mathcal{M}$
if $(U,\theta)$ becomes a coordinate system on~$\mathcal{M}$, $T,c$ are continuous on~$\mathcal{X}$, $\psi$ is smooth on $\mathcal{M}$
and,  for each $p\in U$, the relation~\eqref{exppp} holds on the whole of $x\in \mathcal{X}$.
\item
An $(h,\tau)$-exponential representation $(U,\theta, T, c,\psi)$ on $\mathcal{M}$ is \emph{minimal}
if its order is minimum among the order of any $(h,\tau)$-exponential representations $(\widetilde{U},\tilde{\theta}, \widetilde{T}, \tilde{c},\tilde{\psi})$ on~$\mathcal{M}$
such that $U \cap \widetilde{U}\neq \emptyset$.
\item
We say that $\mathcal{M}$ is an \emph{$(h,\tau)$-exponential family}
if $\mathcal{M}$ is a statistical model where each point in $\mathcal{M}$ has an $(h,\tau)$-exponential coordinate representation. 
\end{enumerate}
In the case of $(h,\tau)=( r\log r, \mathrm{id}_{I})$, 
the terminology \emph{$(h,\tau)$-exponential} is abbreviated to merely \emph{exponential}.
We sometimes abbreviate the terminology \emph{$(h,\tau)$-exponential} as $\exp_{h,\tau}$-
and use  \emph{deformed exponential} as a generic term for it.
\end{definition}
Note that 
\[
\exp_{h,\tau}'=\chi_{h,\tau}\circ \exp_{h,\tau} \quad \text{on }\ell_{h,\tau}(I).
\]

Definition~\ref{exp} is based on the definition of an exponential family due to Barndorff-Nielsen~\cite{BN1970}*{Chapter~5}. 
In the case that $(h,\tau, I) \not\simeq ( r\log r, \mathrm{id}_{I}, I )$, we should discuss a way of normalization.
Bashkirov~\cite{Bashkirov} proposed two ways of normalization, so-called \emph{$Z$-form} and \emph{$S$-form}, 
for R\'enyi distributions and Tsallis distributions.
On one hand, the $Z$-form uses the partition function.
On the other hand, the $S$-form is compatible with the minimality of deformed exponential representations and 
we adopt the $S$-form in Definition~\ref{exp}.
Since the deformed exponential normalization $\psi$ is determined from the  other ingredients of a deformed exponential representation $(U,\theta, T, c,\psi)$, 
it is enough to use the quadruple $(U,\theta, T, c)$ instead of the quintuple $(U,\theta, T, c,\psi)$.
However we clarify $\psi$ since the deformed exponential normalization is strongly related to a geometric structure, as seen in Theorem~\ref{mainthm2}.
Note that, for a deformed exponential coordinate representation $(U,\theta, T, c,\psi)$ on a statistical model in $\mathcal{P}_I(\mu)$, 
$\psi\in C^{\infty}(U)$ holds by Condition \eqref{C1}.

Throughout this section, we fix $(h,\tau,I)\in \mathcal{G}$.
We use  $\mathcal{P}$ and $\mathcal{M}$  to denote a set and an $n$-dimensional manifold in $\mathcal{P}_I(\mu)$, respectively.
\begin{remark}\label{remm}
\begin{enumerate}
\setlength{\leftmargin}{23pt}
\setlength{\leftskip}{-15pt}
\item
For an $\exp_{h,\tau}$-representation $(U,\theta, T, c,\psi)$ on  $\mathcal{P}$,
the positivity and the  $\mu$-integrability of $p\in U$ yields 
\[
 \langle \theta(p), T(x) \rangle -c(x)-\psi(p) \in \ell_{h,\tau}(I)
\quad \text{for $\mu$-a.e.}\, x\in {\mathcal{X}}.
 \]
\item\label{remm2}
 \setlength{\leftskip}{-15pt}
If $\mathcal{M}$ has  an atlas of $\exp_{h,\tau}$-coordinate representations, then $p\in C(\mathcal{X})$ holds for $p\in \mathcal{M}$.
Moreover, if $\mathcal{X}$ is compact, 
then $\mathbb{I}_{\phi}\in C^{\infty}(\mathcal{M})$ holds for $\phi\in C^\infty(I)$.
\item
If $\mathcal{M}$ is an $\exp_{h,\tau}$-family
then so is an $\exp_{\tilde{h}, \tilde{\tau}}$-family 
for $(h,\tau, I)\simeq (\tilde{h}, \tilde{\tau},I)$ since 
it follow from Lemma~\ref{equivclass} that 
there exist $a_1 \in \mathbb{R}$ and $\lambda>0$ that
\begin{align}\label{sim} 
\ell_{h,\tau}(t)=\lambda\ell_{\tilde{h}, \tilde{\tau}}(t)+a_1 \quad\text{for } t\in I,
\end{align}
or equivalently 
\[
\exp_{h,\tau}(u)=\exp_{\tilde{h}, \tilde{\tau}}(\lambda^{-1}(u-a_1))\quad \text{for }u\in \mathbb{R}
\]
Note that the relation~\eqref{sim} does note imply if $(h,\tau, I)\simeq (\tilde{h}, \tilde{\tau},I)$.
\item
If ${\mathcal{X}}$ is a finite set and $\mu$ is the counting measure on~${\mathcal{X}}$, 
then $\mathcal{P}(\mu)$ is an $\exp_{h,\tau}$family for any $(h,\tau, I)$ with $(0,1)\subset I$.
\end{enumerate}
\end{remark}
A similar argument in \cite{BN1970} yields the following.
We omit the proof.
\begin{proposition}{\rm (cf.\cite{BN1970}*{Theorems~5.1, 5.2})}\label{minimal}
For an $\exp_{h,\tau}$-representation $(U, \theta, T, c,\psi)$ of order $N$ on $\mathcal{P}$,
$(U, \theta, T, c,\psi)$ is minimal if and only if both of the following two conditions are valid.
\begin{itemize}
\item
$\theta^1, \ldots, \theta^N$ and $\mathbf{1}_{U}$ are linearly independent.
\item
$T_1, \ldots, T_N$ and $\mathbf{1}_{{\mathcal{X}}}$ are linearly independent.
\end{itemize}
For $j=1,2$, let $(U_j,\theta_j, T^j, c^j,\psi_j)$ be a minimal $\exp_{h,\tau}$-representation on $\mathcal{P}$ such that $U_1 \cap U_2 \neq \emptyset$.
Then there exist $v_1,v_2\in \mathbb{R}^N$ and $A\in \mathrm{GL}(N,\mathbb{R})$ such that 
\begin{align}\label{overlap}
\begin{aligned}
&\theta_1=A\theta_2+v_1, 
&&\psi_1=\psi_2-\langle \theta_2, v_2\rangle && \text{on $U_1 \cap U_2$,}\\
&T^2={}^{\mathsf{T}\!}AT^1+v_2,
&& c^2=c^1-\langle v_1, T^1\rangle 
&&\text{$\mu$-a.e.}
\end{aligned}
\end{align}
\end{proposition}
For a deformed exponential representation $(U,\theta, T, c,\psi)$, we show the injectivity of~$\theta$.
In the setting of Barndorff-Nielsen~\cite{BN1970}*{Chapter~5}, 
$\theta$ is not assumed to be injective
and $(U,\theta, T, c,\psi)$ is said to be \emph{canonical} if $\theta$  is injective.
\begin{lemma}\label{ordn}
For an $\exp_{h,\tau}$-representation $(U,\theta, T, c,\psi)$ on $\mathcal{M}$, 
$\theta$ is injective on $U$. 
Moreover, then order of $(U,\theta, T, c,\psi)$ is at least $n$ if $\theta$ is continuous on $U$.
\end{lemma}
\begin{proof}
For $p,p'\in U$, assume 
\[
\theta_\ast\coloneqq \theta(p)=\theta(p'),
\qquad
\psi(p) \leq \psi(p'). 
\]
Since $\exp_{h,\tau}$ is an increasing function, $p \geq p'$ holds $\mu$-a.e., it follows
\[
0\leq 
\int_{{\mathcal{X}}} \left( p(x)-p'(x) \right) d\mu(x)
=
\int_{{\mathcal{X}}} p(x)d\mu(x)-\int_{{\mathcal{X}}} p'(x) d\mu(x)
=1-1=0.
\]
This leads to 
$\psi(p)=\psi(p')$ hence $p = p'$ $\mu$-a.e.,
that is $\theta$ is injective.

Assume the continuity of $\theta$ on $U$ and the order $N$ of $(U,\theta, T, c,\psi)$ is less than $n$.
Denote by $i_N$ the inclusion map from $\mathbb{R}^N$ to $\mathbb{R}^{n}$.
Let $(U',\xi)$ be a coordinate system on $\mathcal{M}$ with $U\cap U'\neq \emptyset$.
Since $i_{N}\circ \theta \circ \xi^{-1}\colon \xi(U\cap U')\to \mathbb{R}^{n}$ is injective and continuous,
the invariance of domain theorem implies that $i_{N}(\theta(U\cap U'))$ is open in $\mathbb{R}^{n}$,
which is a contradiction.
\end{proof}

\begin{proposition}\label{global}
If $\mathcal{M}$ has  an atlas of $\exp_{h,\tau}$-coordinate representations, 
then $\mathcal{M}$ has a minimal global $\exp_{h,\tau}$-coordinate representation.
\end{proposition}
\begin{proof}
Since each covering of $\mathcal{M}$ has a countable locally finite open refinement due to the paracompactness of $\mathcal{M}$,
to prove the existence of a global $\exp_{h,\tau}$-coordinate representation,
it is enough to show that two overlapping $\exp_{h,\tau}$-coordinate representations can be merged.
For $j=1,2$, let $(U_j, \theta_j, T^j, c^j,\psi_j)$ be an $\exp_{h,\tau}$-coordinate representation on $\mathcal{M}$
such that $U_1 \cap U_2 \neq \emptyset$.
By Proposition~\ref{minimal}, there exist $v_1,v_2\in \mathbb{R}^n$ and $A\in \mathrm{GL}(n,\mathbb{R})$ such that~\eqref{overlap} holds.
Set $U\coloneqq U_1\cup U_2$.
If we define $\theta\colon U \to \mathbb{R}^n$ and $\psi\colon U\to \mathbb{R}$ respectively by 
\[
\theta(p)\coloneqq \begin{cases}\theta_1(p) &\text{if }p\in U_1, \\ A\theta_2(p) +v_1 &\text{if }p\in U\setminus U_1, \end{cases}
\qquad
\psi(p)\coloneqq \begin{cases}\psi_1(p) &\text{if }p\in U_1, \\ \psi_2(p)-\langle \theta_2(p), v_2\rangle &\text{if }p\in U\setminus U_1,\end{cases}
\]
then $\theta$ and $\psi$ are smooth on $U$.
Moreover, $\theta$ is injective on $U$.
Indeed, if there exists $p_j \in U_j \setminus ( U_1 \cap U_2)$ for $j=1,2$ such that $\theta(p_1)=\theta(p_2)$, 
then we have
\begin{align*}
 p_1(x)
 &=\exp_{h,\tau} \left( \langle \theta(p_1), T^1(x) \rangle-c^1(x)-\psi_1(p_1)\right) \\
 &=\exp_{h,\tau} \left( \langle \theta(p_2), T^2(x) \rangle-c^2(x)-\left( \psi_1(p_1)+\langle\theta_2(p_2), v_2\rangle \right) \right), \\
 p_2(x)
 &=\exp_{h,\tau} \left( \langle \theta(p_2), T^2(x) \rangle-c^2(x)-\psi_2(p_2) \right)
\end{align*}
for $x\in \mathcal{X}$.
The monotonicity of $\exp_{h,\tau}$ together with the condition $p_1,p_2\in \mathcal{P}(\mu)$ leads to $p_1=p_2$,
which contradicts the assumption that $p_j \in U_j \setminus ( U_1 \cap U_2)$.
Thus $\theta$ is injective on $U$ and $(U, \theta, T^1,c^1,\psi)$ is an $\exp_{h,\tau}$-coordinate representation on~$\mathcal{M}$.
Its minimality immediately follows from Lemma~\ref{ordn} and the proof is archived.
\end{proof}
We consider a condition for a manifold in $\mathcal{P}(\mu)$ to be a statistical model.
\begin{lemma}\label{lem}
Let $(\mathcal{M},\theta,T,c,\psi)$ be  an $\exp_{h,\tau}$-coordinate representation on $\mathcal{M}$.
\begin{enumerate}
 \setlength{\leftskip}{-15pt}
 \item\label{461}
$\mathcal{M}$ satisfies Conditions~\eqref{C5}--\eqref{C3}.
\item\label{462}
 If either $\mathcal{X}$ is compact or $(h,\tau)= (r\log r, \mathrm{id}_I)$,  then $\mathcal{M}$ satisfies Condition~\eqref{C4}.
\item\label{463}
Assume Condition~\eqref{C4}.
If $\mathbb{I}_{\chi_{h,\tau}}\in \mathbb{R}^\mathcal{M}$,
then, for $p\in \mathcal{M}$ and $1\leq i,j\leq n$, 
\begin{align}\label{derideri}
\begin{split}
\partial_i\psi |_p 
=\frac{\mathbb{I}_{T_i \chi_{h,\tau}}(p)}{\mathbb{I}_{\chi_{h,\tau}}(p)},\qquad
\partial_i\partial_j\psi |_p 
=
\frac{1}{ \mathbb{I}_{\chi_{h,\tau}}(p)}\cdot\int_{{\mathcal{X}}} \partial_i e_x|_p \cdot \partial_j e_x|_p \cdot 
\frac{ \chi_{h,\tau}'(p(x))}{\chi_{h,\tau}(p(x))} d\mu(x).
\end{split}
\end{align}
\end{enumerate}
\end{lemma}
\begin{proof}
Since $T$ and $c$ are continuous on $\mathcal{X}$, 
we have 
\[
 \langle \theta(p), T(x) \rangle -c(x)-\psi(p) \in \ell_{h,\tau}(I) 
 \quad \text{for }x\in \mathcal{X}.
\] 
This with the smoothness of $\theta, \psi$ and  $\exp_{h,\tau}$ ensures that, for each $x\in \mathcal{X}$, 
\[
p\mapsto e_x(p)=\exp_{h,\tau} \left(\langle \theta(p), T(x) \rangle -c(x)-\psi(p) \right)
\]
is smooth on $\mathcal{M}$.
Moreover,  its derivatives of any order at $p\in \mathcal{M}$ is continuous, in particular Borel on $\mathcal{X}$.
Thus~$\mathcal{M}$ satisfies Conditions~\eqref{C1} and \eqref{C2}.
By the relation $\exp'_{h,\tau} =\chi_{h,\tau} \circ \exp_{h,\tau}$, we have   
\begin{align}\label{diff}
\begin{split}
\partial_i e_x
&=\left(T_i(x)-\partial_i \psi \right)\cdot (\chi_{h,\tau} \circ e_x),\\
\partial_i\partial_j e_x
&=-\partial_i \partial_j\psi \cdot (\chi_{h,\tau} \circ e_x)
+
\partial_i e_x \cdot
\partial_j e_x \cdot
\frac{\chi_{h,\tau}'}{\chi_{h,\tau}}\circ e_x.
\end{split}
\end{align}
which leads to
\[
V_p(e_x)
=
\sum_{i=1}^{n} V_p^i \left(T_i(x)-\partial_i\psi|_p \right) \chi_{h,\tau}(p(x))
\quad
\text{for $V\in \mathfrak{X}(\mathcal{M})$ and $p\in \mathcal{M}$}.
\]
Since $\chi_{h,\tau}>0$ on $I$ and $T_1, \ldots, T_n$ and $\mathbf{1}_{{\mathcal{X}}}$ are linearly independent by Proposition~\ref{minimal},
if $V_p(e_x)=0$ for $\mu$-a.e.\,${x}\in {\mathcal{X}}$ (hence all $x\in \mathcal{X}$), then $V_p^i=0$ holds for each $i$, that is, Condition \eqref{C3} holds.
Let 
$\mathcal{U}_\mathcal{M}$ be the topology of the manifold $\mathcal{M}$
and 
$\mathcal{U}$ be the induced topology of the set $\mathcal{M}$ by the family of functions  $(e_x)_{x\in \mathcal{X}}$, respectively.
Since $e_x$ is continuous on $\mathcal{M}$  for each $x\in \mathcal{X}$, 
we have $\mathcal{U}\subset \mathcal{U}_{\mathcal{M}}$.
For $U\in \mathcal{U}_{\mathcal{M}}$ and $p\in U$, there exists $\varepsilon>0$ such that 
\[
\{ p'\in \mathcal{M}\mid |\theta(p')-\theta(p)|<\varepsilon  \}\subset U.
\]
For  $x\in \mathcal{X}$ and $\delta>0$, setting 
\[
U_{x,\delta}\coloneqq (\ell_{h,\tau}\circ e_x)^{-1}\left(\ell_{h,\tau}( p(x))-\delta,  \ell_{h,\tau}(p(x))+\delta   \right)\in \mathcal{U},
\] 
we have
\[
|\ell_{h,\tau}(p'(x))-\ell_{h,\tau}(p(x))|
=\left| \langle \theta(p')-\theta(p),T(x)\rangle-(\psi(p')-\psi(p))\right|<\delta
\quad\text{for }p'\in U_{x,\delta}.
\]
Since $T_1, \ldots, T_n$ and $\mathbf{1}_{{\mathcal{X}}}$ are linearly independent by Proposition~\ref{minimal},
it holds that 
\[
\limsup_{\delta \downarrow 0}\sup_{p'\in U_{x,\delta}} |\theta(p')-\theta(p)|=0.
\]
This implies that there exists $\delta>0$ such that
\[
U_{x,\delta}\subset \{ p'\in \mathcal{M}\mid |\theta(p')-\theta(p)|<\varepsilon  \}\subset U
\]
hence $U\in \mathcal{U}$.
Thus Claim~\eqref{461} follows.

To prove Claim~\eqref{462}, 
it is enough to show Condition~\eqref{C4} for $V=\partial_i, W=\partial_j$.
On one hand, 
if $\mathcal{X}$ is compact, then $\mu(\mathcal{X})<\infty$ and \eqref{diff} together with the dominated convergence theorem implies Condition~\eqref{C4}.
On the other hand, assume $(h,\tau)= (r\log r, \mathrm{id}_I)$.
For $p\in \mathcal{M}$, there exists $\delta>0$ such that $\theta(p)+\varepsilon e_i \in \theta(\mathcal{M})$ for $\varepsilon \in (-2\delta,2\delta)$.
Setting $p_\varepsilon \coloneqq \theta^{-1}(\theta(p)+\varepsilon e_i)$,
for $0<|\varepsilon|<\delta$, we have
\begin{align*}
\int_{\mathcal{X}}\frac{e_x(p_\varepsilon)-e_x(p)}{\varepsilon}d\mu(x)
&=
\int_{\mathcal{X}}p(x)\frac{e^{\varepsilon T_i(x)+\psi(p)-\psi(p_\varepsilon)}-1}{\varepsilon}d\mu(x),\\
\left|\frac{e^{\varepsilon T_i(x)+\psi(p)-\psi(p_\varepsilon)}-1}{\varepsilon}\right|
&\leq 
\left|\frac{e^{\varepsilon T_i(x)}-1}{\varepsilon}\right|e^{\psi(p)-\psi(p_\varepsilon)}
+
\left|\frac{e^{-\psi(p_\varepsilon)}-e^{-\psi(p)}}{\varepsilon}\right|e^{\psi(p)}\\
&\leq 
\frac{e^{\delta T_i(x)}+e^{-\delta T_i(x)}}{\delta}\max_{\epsilon \in [-\delta,\delta]} e^{\psi(p)-\psi(p_{\epsilon})}
+
\max_{\epsilon \in [-\delta,\delta]} \left|e^{-\psi(p_\epsilon)}\partial_i \psi|_{p_{\epsilon}} \right| e^{\psi(p)}\\
&\leq  
e^{\psi(p)}\cdot \max_{\epsilon \in [-\delta,\delta]} e^{-\psi(p_{\epsilon})}
\left(\frac{e^{\delta T_i(x)}+e^{-\delta T_i(x)}}{\delta}
+
\max_{\epsilon \in [-\delta,\delta]} \left|\partial_i \psi|_{p_{\epsilon}} \right|
\right).
\end{align*}
Since the right-hand side of
\begin{align*}
&p(x)\cdot e^{\psi(p)} \cdot \max_{\epsilon \in [-\delta,\delta]} e^{-\psi(p_{\epsilon})}
\left(\frac{e^{\delta T_i(x)}+e^{-\delta T_i(x)}}{\delta}
+
\max_{\epsilon \in [-\delta,\delta]} \left|\partial_i \psi|_{p_{\epsilon}} \right|\right)\\
\leq&
\max_{\epsilon,\epsilon' \in [-\delta,\delta]} e^{\psi(p_{\epsilon'})-\psi(p_{\epsilon})}\cdot \left(
\frac{p_\delta(x)+p_{-\delta}(x)}{\delta} +p(x) \cdot 
\max_{\epsilon \in [-\delta,\delta]} \left| \partial_i \psi|_{p_\epsilon}\right|
\right)
\end{align*}
is $\mu$-integrable, 
the dominated convergence theorem implies
\begin{align*}
0=\partial_i|_p \int_{\mathcal{X}} e_x \mu(x)=\int_{\mathcal{X}} \partial_i e_x |_pd\mu(x).
\end{align*}
Similarly, we have
\[
0=\partial_i|_p \int_{\mathcal{X}} \partial_j e_x \mu(x)=\int_{\mathcal{X}} \partial_i \partial_j e_x |_p d\mu(x),
\]
and Claim~\eqref{462} holds.

Finally, we prove Claim~\eqref{463}.
Assume Condition~\eqref{C4}.
We observe from \eqref{diff} that
\begin{align*}
0
&=\int_{{\mathcal{X}}} \partial_i e_x |_p d\mu(x)
=\int_{{\mathcal{X}}} \left(T_i(x)-\partial_i \psi|_p \right) \chi_{h,\tau}(p) d\mu,\\
0
&=\int_{{\mathcal{X}}} \partial_i\partial_j e_x |_p d\mu(x)
=
\int_{{\mathcal{X}}}
\left(-\partial_i \partial_j\psi|_p \cdot \chi_{h,\tau}(p(x)) + \partial_i e_x|_p \cdot \partial_j e_x|_p \cdot 
\frac{ \chi_{h,\tau}'(p(x))}{\chi_{h,\tau}(p(x))}\right) d\mu(x).
\end{align*}
If $\mathbb{I}_{\chi_{h,\tau}}\in \mathbb{R}^\mathcal{M}$,
then \eqref{derideri} holds, that is, Claim~\eqref{463} is proved.
\end{proof}
By the proof of Lemma~\ref{lem}\eqref{462}, 
if $(h, \tau)=(r\log r, \mathrm{id}_I)$,
then we have $\psi\in C^{\infty}(\mathcal{M})$ and $\mathbb{I}_{T}, \mathbb{I}_{\langle T, T\rangle}\in \mathbb{R}^{\mathcal{M}}$,
which is mentioned in~\cite{BN1970}*{p.5.6}.
Moreover, it is stated in \cite{BN1970}*{Theorem~5.5} that $(\mathcal{M},\mathbb{I}_T)$ becomes a coordinate system
since  $(\partial_i \partial_j \psi)_{1\leq i,j\leq n}$ is positive definite on $\mathcal{M}$.
This positivity also holds for a deformed exponential family.
\begin{corollary}\label{psipos}
Let  $\mathcal{M}$ be an $\exp_{h,\tau}$-family equipped with 
an $\exp_{h,\tau}$-coordinate representation $(\mathcal{M},\theta,T,c,\psi)$.
If  $\mathbb{I}_{\chi_{h,\tau}}\in \mathbb{R}^\mathcal{M}$and $\chi'_{h,\tau}>0$ on $I$, 
then $(\partial_i \partial_j \psi)_{1\leq i,j\leq n}$ is positive definite on~$\mathcal{M}$. 
\end{corollary}
\begin{proof}
By \eqref{derideri}, it is enough to show that 
\begin{equation}\label{pos}
\int_{{\mathcal{X}}} \partial_i e_x|_p \cdot \partial_j e_x|_p \cdot 
\frac{ \chi_{h,\tau}'(p(x))}{\chi_{h,\tau}(p(x))} d\mu(x)
=
\int_{{\mathcal{X}}} (T_i-\partial_i\psi|_p)(T_j-\partial_j\psi|_p) \chi_{h,\tau}(p) \chi_{h,\tau}'(p)d\mu
\end{equation}
determines a  positive definite symmetric matrix.
Since $T_1, \ldots, T_N$ and $\mathbf{1}_{{\mathcal{X}}}$ are linearly independent by Proposition~\ref{minimal},
$((T_i-\partial_i\psi|_p )(T_j-\partial_j\psi|_p ) )_{1\leq i,j\leq n}$ is positive definite on~$\mathcal{X}$.
This with the positivity of $\chi_{h,\tau}\chi'_{h,\tau}$ on $I$ ensures
the positive definiteness of the symmetric matrix determined by~\eqref{pos}.
\end{proof}
\section{Proof of Theorem~\ref{mainthm2}}\label{pfthm2}
Throughout this section, we fix $(h,\tau,I)\in \mathcal{G}$ 
and let $\mathcal{M}$ be an $n$-dimensional statistical model   in $\mathcal{P}_I(\mu)$.
To prove Theorem~\ref{mainthm2}, 
we discuss a condition for~$(h,\tau,I)$ such that the $(h,\tau)$-divergence on $\mathcal{M}$ induces a Riemannian metric and a connection via \eqref{met}.
\begin{definition}\label{compatible}
We say that $\mathcal{M}$ is \emph{$(h,\tau)$-compatible} 
if, for $V,W,Z\in \mathfrak{X}(\mathcal{M})$, 
\begin{align*}
\widetilde{V}_{(p, p)}( \widehat{W} \left(d_{h,\tau}\circ (e_x, e_x) \right), 
\widehat{V}_{(p, p)}(\widehat{W}\widetilde{Z} \left(d_{h,\tau}\circ (e_x, e_x)\right)
\in L^1(\mu)
\end{align*}
holds for $p\in \mathcal{M}$ 
and the functions 
\begin{align*}
p\mapsto g^{h,\tau}_p(V,W)
&\coloneqq-\int_{{\mathcal{X}}}\widetilde{V}_{(p, p)}( \widehat{W} \left(d_{h,\tau}\circ (e_x, e_x) \right) d\mu(x),\\
p\mapsto g^{h,\tau}_p(\nabla^{h,\tau}_ZV,W)
&\coloneqq -\int_{{\mathcal{X}}}
\widehat{Z}_{(p, p)}(\widehat{V}\widetilde{W} \left(d_{h,\tau}\circ (e_x, e_x)\right)
d\mu(x)
\end{align*}
are smooth on $\mathcal{M}$.
\end{definition}
If $\mathcal{M}$ is $(h,\tau)$-compatible,
then so is $(\tilde{h},\tilde{\tau})$-compatible for $(h,\tau,I)\simeq (\tilde{h},\tilde{\tau},I)$.
The manifold of Gaussian densities  is $(h,\tau)$-compatible 
for $(h,\tau,(0,\infty))\simeq (r\log r, \mathrm{id}_{(0,\infty)}, (0,\infty))$.
In the case that ${\mathcal{X}}$ is a finite set and $\mu$ is the counting measure on ${\mathcal{X}}$,
then $\mathcal{P}(\mu)$ is $(h,\tau)$-compatible for any $(h,\tau,I)$ if $(0,1)\subset I$.
Note that, even on an $(h,\tau)$-compatible statistical model, the $(h,\tau)$-divergence is not necessarily a contrast function.

\begin{remark}\label{Fisher}
In the case of $(h,\tau)=(r\log r, \mathrm{id}_{I})$,
$g^{h,\tau}$ and $\nabla^{h,\tau}$ are called the \emph{Fisher metric}
and the \emph{exponential connection},
respectively.
\end{remark}

We show that $g^{h,\tau}$ and $\nabla^{h,\tau}$ are indeed a Riemannian metric and a connection.
\begin{lemma}\label{metcon}
If $\mathcal{M}$ is $(h,\tau)$-compatible,
then $g^{h,\tau}$ is a Riemannian metric and $\nabla^{h,\tau}$ is a torsion-free connection on~$\mathcal{M}$, respectively.
\end{lemma}
\begin{proof}
Fix $V,W,Z \in \mathfrak{X}(\mathcal{M})$ and $p\in \mathcal{M}$.
A direct calculation gives
\begin{align}\label{cal}
\begin{split}
-\widetilde{V}_{(p, p)} \left( \widehat{W} \left(d_{h,\tau}\circ (e_x, e_x) \right) \right)
=&
 V_p(e_x) W_p(e_x) m_{h,\tau} (p(x)),\\
-
\widehat{Z}_{(p,p)}\left(\widehat{V}\widetilde{W} \left(d_{h,\tau}\circ (e_x, e_x) \right)\right)
=&
Z_p(e_x)
V_p(e_x)
W_p(e_x) \gamma_{h,\tau}(p(x))\\
&+ Z_p(V (e_x)) W_p(e_x) m_{h,\tau}(p(x)),
\end{split}
\end{align}
implying $g^{h,\tau}_p(V,W)=g^{h,\tau}_p(W,V)$.
For $\varphi\in C^{\infty}(\mathcal{M})$,
it turns out that 
\begin{align*}
g^{h,\tau}_p(\varphi V,W)
=\int_{\mathcal{X}} (\varphi V)_p(e_x) W_p(e_x) m_{h,\tau} (p(x)) d\mu(x)
=\varphi (p)\cdot g^{h,\tau}_p(V,W),
\end{align*}
which ensures that $g^{h,\tau}$ is a symmetric $(0,2)$ tensor filed on $\mathcal{M}$.
Since $m_{h,\tau}$ is positive on~$I$, it follows from Condition \eqref{C3} that $g_p^{h,\tau}(V,V)>0$ if $V_p \neq 0$.
Similarly, we find
\begin{align*}
g^{h,\tau}_p(\nabla^{h,\tau}_{\varphi Z}V,W)
&=
\varphi (p) \cdot g^{h,\tau}_p(\nabla^{h,\tau}_ZV,W), \\
g^{h,\tau}_p(\nabla^{h,\tau}_Z \varphi V,W)
&=
\varphi (p) \cdot g^{h,\tau}_p(\nabla^{h,\tau}_ZV,W)+Z_p(\varphi) g^{h,\tau}_p(V,W),\\
g^{h,\tau}_p(\nabla^{h,\tau}_ZV,W)
-g^{h,\tau}_p(\nabla^{h,\tau}_ZV,W)
&=g^{h,\tau}_p([Z,V],W),
\end{align*}
implying that $\nabla^{h,\tau}$ is a torsion-free connection on $\mathcal{M}$.
Thus the proof is achieved.
\end{proof}
For $(h|_{\tau(I)},\tau, I)\in \mathcal{G}_\ast$,  
if $\mathcal{M}$ is $(h|_{\tau(I)},\tau)$-compatible and $(h^\ast|_{\ell(I)},\ell)$-compatible,
then $g^{h|_{\tau(I)},\tau}=g^{h^\ast|_{\ell(I)},\ell}$ but $\nabla^{h|_{\tau(I)},\tau}\neq \nabla^{h^\ast|_{\ell(I)},\ell}$ in general
(see Remark~\ref{underinv} and also Remark~\ref{ambiguity}\eqref{gfm}).

If $\mathcal{M}$ has  an  $\exp_{h,\tau}$-coordinate representation $(\mathcal{M},\theta,T,c,\psi)$,
then it follows from~\eqref{diff} and \eqref{cal} that
\begin{align}\label{third}
\begin{split}
-\widetilde{\partial_i}\left( \widehat{\partial_j} \left(d_{h,\tau}\circ (e_x, e_x) \right) \right)\Bigm|_{(p, p)} 
=&
(T_i-\partial_i \psi|_p) \cdot 
(T_j-\partial_j \psi|_p) \cdot \tau'(p(x))\chi_{h,\tau}(p(x)),\\ 
-
\widehat{\partial_i}\left(\widehat{\partial_j}\widetilde{\partial_k} \left(d_{h,\tau}\circ (e_x, e_x) \right)\right)\Bigm|_{(p, p)}
=&-\partial_i \partial_j \psi|_p \cdot \partial_k (\tau \circ e_x)|_p,
\end{split}
\end{align}
where we used the properties
\[
\chi_{h,\tau}\cdot m_{h,\tau}=\tau',\qquad
\gamma_{h,\tau}
+\frac{\chi'_{h,\tau}}{\chi_{h,\tau}}m_{h,\tau}
=0,
\quad \text{on $I$}.
\]
This gives the following consequence.
\begin{lemma}\label{lemm}
Assume that $\mathcal{M}$ has  an  $\exp_{h,\tau}$-coordinate representation $(\mathcal{M},\theta,T,c,\psi)$.
If either $\mathcal{X}$ is compact or $(h,\tau)=(r\log r, \mathrm{id}_I)$,
then~$\mathcal{M}$ is $(h,\tau)$-compatible.
\end{lemma}
\begin{proof}
If $\mathcal{X}$ is compact, 
then $\mu(\mathcal{X})<\infty$ ensures the integrability in Definition~\ref{compatible}.
Moreover,  the smoothness  in Definition~\ref{compatible} follows from the dominated convergence theorem.
On the other hand, if $(h,\tau)=(r\log r, \mathrm{id}_I)$, then the claim follows from 
$\chi_{h,\tau}=\mathrm{id}_I$ and~\eqref{derideri} together with~\eqref{third}.
\end{proof}

We give a condition for $( g^{h,\tau},\nabla^{h,\tau})$ being a Hessian structure on an $\exp_{h,\tau}$-family.
Recall
\[
s^{\star}_{h,\tau}\coloneqq-\tau \cdot (h'\circ \tau)+(h\circ \tau)=-\tau \cdot \ell_{h,\tau}+(h\circ \tau).
\]
If $(h,\tau)=(r\log r, \mathrm{id}_I)$, then $s^{\star}_{h,\tau}=-\mathrm{id}_{I}$.
\begin{theorem}\label{mainthm22}
Let  $\mathcal{M}$ be an $(h,\tau)$-compatible $\exp_{h,\tau}$-family.
Assume  $\mathbb{I}_{\chi_{h,\tau}}\in \mathbb{R}^\mathcal{M}$ and $\mathbb{I}_\tau \in C^{\infty}(\mathcal{M})$ such that 
\begin{align}\label{eexchange}
V_p(\mathbb{I}_\tau)&=\int_{{\mathcal{X}}} V_p( \tau \circ e_x ) d\mu(x)
\quad\text{for }V\in \mathfrak{X}(\mathcal{M}) \text{ and }p\in \mathcal{M}.
\end{align}
\begin{enumerate}
 \setlength{\leftskip}{-15pt}
\item\label{541}
For an $\exp_{h,\tau}$-coordinate representation $(\mathcal{M},\theta,T,c, \psi)$ on $\mathcal{M}$,
the coordinate system~$(\mathcal{M},\theta)$ is affine with respect to $\nabla^{h,\tau}$ if $\mathbb{I}_\tau$ is constant on $\mathcal{M}$. 
Moreover, if $\chi_{h,\tau}'>0$ on $I$ and $(\mathcal{M},\theta)$ is affine with respect to $\nabla^{h,\tau}$, 
then $\mathbb{I}_\tau$ is constant on $\mathcal{M}$.
\item\label{542}
Assume that $\mathbb{I}_\tau$ is constant on $\mathcal{M}$ and $\mathbb{I}_{s^\star_{h,\tau}} \in C^\infty(\mathcal{M})$ satisfies 
\begin{align}\label{exchange1}
V_p(\mathbb{I}_{s^\star_{h,\tau}})
&=\int_{\mathcal{X}} V_p \left(s^\star_{h,\tau} \circ e_x \right) d\mu(x),\\ \label{exchange2}
V_p\left(W (\mathbb{I}_{s^\star_{h,\tau}})\right)
&=\int_{\mathcal{X}} V_p\left(W \left(s^\star_{h,\tau} \circ e_x \right)\right) d\mu(x),
\end{align}
for $V,W \in \mathfrak{X}(\mathcal{M})$ and $p\in \mathcal{M}$.
Then $( g^{h,\tau},\nabla^{h,\tau})$ is a Hessian structure on $\mathcal{M}$ and its global potential is $-\mathbb{I}_{s^\star_{h,\tau}}+\psi\mathbb{I}_\tau$, 
where $\psi$ is a global $\exp_{h,\tau}$-normalization of $\mathcal{M}$.
Moreover, the canonical divergence of the Hessian structure $( g^{h,\tau},\nabla^{h,\tau})$ on $\mathcal{M}$ is the $(h,\tau)$-divergence.
\end{enumerate}
\end{theorem}
\begin{proof}
Let $(\mathcal{M},\theta,T,c, \psi)$ be an $\exp_{h,\tau}$-coordinate representation on $\mathcal{M}$.
By~\eqref{third}, we have
\begin{equation}\label{flat}
g^{h,\tau}(\nabla^{h,\tau}_{\partial_i}\partial_j, \partial_k)
=\int_{\mathcal{X}} \left(-\partial_i \partial_j \psi|_p \cdot \partial_k (\tau \circ e_x)|_p \right)d\mu(x)
=-\partial_i \partial_j \psi|_p \cdot \partial_k \mathbb{I}_\tau|_p
\end{equation}
for $1\leq i,j,k \leq n$.
Thus if $\mathbb{I}_\tau$ is constant on $\mathcal{M}$,
then $(\mathcal{M},\theta)$ is affine with respect to $\nabla^{h,\tau}$. 
Moreover, if $\chi_{h,\tau}'>0$,  then  $\partial_i^2 \psi >0$ on $\mathcal{M}$ by Corollary~\ref{psipos}.
This implies that if $(\mathcal{M},\theta)$ is affine with respect to $\nabla^{h,\tau}$ 
then  $\partial_k\mathbb{I}_\tau=0$ on $\mathcal{M}$ for each $k$, that is,  $\mathbb{I}_\tau$ is constant on~$\mathcal{M}$.
Thus Claim~\eqref{541} holds.

Let us prove Claim~\eqref{542}.
Assume that $\mathbb{I}_\tau$ is constant on $\mathcal{M}$  and $\mathbb{I}_{s^\star_{h,\tau}} \in C^\infty(\mathcal{M})$ satisfies~\eqref{exchange1},\eqref{exchange2}.
Fix an $\exp_{h,\tau}$-coordinate representation $(\mathcal{M},\theta,T,c, \psi)$ on $\mathcal{M}$,
which is unique  up to affine transformation by Proposition~\ref{minimal}.
Then $(\mathcal{M},\theta)$ is affine with respect to $\nabla^{h,\tau}$ by Claim~\eqref{541}
and $\nabla^{h,\tau}$ is flat on $\mathcal{M}$ by Proposition~\ref{straightforward}.
Moreover,
$\nabla^{h,\tau}d \psi$ is determined independent of the choice of global $\exp_{h,\tau}$-normalization $\psi$.
Using  the relation
\[
s^\star_{h,\tau}{}'=- \frac{\tau}{\chi_{h,\tau}},\qquad
s^\star_{h,\tau}{}''=-m_{h,\tau}+ \frac{\tau\chi_{h,\tau}'}{\chi_{h,\tau}^2},
\]
with~\eqref{exchange2} and \eqref{diff}, we have
\begin{align*}
\partial_i \partial_j \mathbb{I}_{s^\star_{h,\tau}}|_p
=&
\int_{{\mathcal{X}}} \partial_i \partial_j \left(s^\star_{h,\tau} \circ e_x\right)|_p
d\mu(x)\\
=&
\int_{{\mathcal{X}}} 
\left(
 \partial_i e_x|_p\cdot\partial_j e_x|_p\cdot s^\star_{h,\tau}{}''(p(x))
 +
\partial_i \partial_j e_x|_p  \cdot s^\star_{h,\tau}{}'(p(x))
\right)d\mu(x)\\
=&
\int_{{\mathcal{X}}} 
\left(
-\partial_i e_x|_p\cdot\partial_j e_x|_p \cdot m_{h,\tau}(p(x))
+ \partial_i\partial_j \psi|_p\cdot \tau(p(x))
\right)d\mu(x)\\
=&
-g^{h,\tau}_{ij}(p)+\partial_i\partial_j \psi|_p \cdot \mathbb{I}_\tau(p).
\end{align*}
By the property~\eqref{flatness}, we find  
\[
g^{h,\tau}=\nabla^{h,\tau} d \left(-\mathbb{I}_{s^\star_{h,\tau}}+ \psi \mathbb{I}_\tau \right)
\]
and $( g^{h,\tau},\nabla^{h,\tau})$ is a Hessian structure on $\mathcal{M}$ with the global potential $-\mathbb{I}_{s^\star_{h,\tau}}+ \psi \mathbb{I}_\tau$.

For $1\leq i \leq n$, 
we observe from~\eqref{exchange1} and \eqref{diff} that 
\begin{align*}
\partial_i \left(-\mathbb{I}_{s^\star_{h,\tau}}+ \psi \mathbb{I}_\tau\right)|_p
&=-\int_{\mathcal{X}} \partial_i e_x|_p \cdot s^\star_{h,\tau}{}'(p(x)) d\mu(x) +\partial_i \psi|_p \cdot \mathbb{I}_\tau(p)\\
&=\int_{\mathcal{X}} (T_i(x)- \partial_i \psi|_p ) \chi_{h,\tau}(p(x))\cdot \frac{\tau(p(x))}{\chi_{h,\tau}(p(x))} d\mu(x) 
+\partial_i \psi|_p \cdot \mathbb{I}_\tau(p)\\
&=\int_{\mathcal{X}} T_i \tau(p) d\mu\\
&=\mathbb{I}_{T_i\tau}(p).
\end{align*}
For $p,p'\in \mathcal{M}$, 
since we have
\begin{equation}\label{aaa}
 \ell_{h,\tau}(p)- \ell_{h,\tau}(p')
 =\langle \theta(p)-\theta(p'), T\rangle-\psi(p)+\psi(p')
\quad \text{on } \mathcal{X},
\end{equation}
it follows from Proposition~\ref{cano} that
the canonical divergence $\overline{D}$ of $(g^{h,\tau}, \nabla^{h,\tau})$ is given by
\begin{align*}
\overline{D}(p,p')
&=
(-\mathbb{I}_{s^\star_{h,\tau}}(p')+ \psi \mathbb{I}_\tau(p'))
-
(-\mathbb{I}_{s^\star_{h,\tau}}(p)+ \psi \mathbb{I}_\tau(p))
+\langle \theta(p)-\theta(p'), \mathbb{I}_{T\tau}(p) \rangle\\
&=-\mathbb{I}_{s^\star_{h,\tau}}(p')+\mathbb{I}_{s^\star_{h,\tau}}(p)
+\int_{\mathcal{X}} (-\psi(p)+\psi(p'))\tau(p)d\mu
+\int_{\mathcal{X}} \left\langle \theta(p) -\theta(p'), T \right\rangle \tau(p)d\mu \\
&=-\mathbb{I}_{s^\star_{h,\tau}}(p')+\mathbb{I}_{s^\star_{h,\tau}}(p)
+\int_{\mathcal{X}} 
 \left( \ell_{h,\tau}(p)- \ell_{h,\tau}(p')\right) \tau(p) d\mu \\
&=\int_{\mathcal{X}}
\left[
h(\tau(p))-h(\tau(p'))-(\tau(p)-\tau(p'))\cdot \ell_{h,\tau}(p')
\right]d\mu\\
&=D_{h,\tau}(p,p'),
\end{align*}
where the second equality follows from the assumption that $\mathbb{I}_\tau$ is constant on $\mathcal{M}$. This completes the proof of Claim~\eqref{542}.
\end{proof}
\begin{proof}[Proof of Theorem~\ref{mainthm2}]
Since $\mathcal{X}$ is compact, 
$\mathcal{M}$ is $(h,\tau)$-compatible by Lemma~\ref{lemm}.
Moreover, $\mathbb{I}_\tau, \mathbb{I}_{s^\star_{h,\tau}} \in C^{\infty}(\mathcal{M})$ and 
\eqref{eexchange}--\eqref{exchange2} hold.
Then the proof immediately follows from Theorem~\ref{mainthm22}.
\end{proof}
\begin{remark}
\begin{enumerate}
\setlength{\leftmargin}{23pt}
\setlength{\leftskip}{-15pt}
\item
The flatness of $\nabla^{h,\tau}$ on an $(h,\tau)$-compatible $\exp_{h,\tau}$-family is independent of the choice of representatives in the equivalence class defined by $\simeq$.
However,
since~$s_{h,\tau}^{\star}$  depends on the choice of representatives, 
whether or not $(g^{h,\tau},\nabla^{h,\tau})$ becomes a Hessian structure on  an $(h,\tau)$-compatible $\exp_{h,\tau}$-family depends on the choice of representatives.
\item
A  claim similar to Theorem~\ref{mainthm22} can be found in \cite{NZ}*{Proposition~2}.
However, in~\cite{NZ}, the notion of Hessian metric is defined without considering a connection,
which differs from our definition of a Hessian structure.
\end{enumerate}
\end{remark}

We give a sufficient condition for $( g^{h,\tau},\nabla^{h,\tau})$ to be $1$-conformally equivalent to a Hessian structure.
\begin{proposition}{\rm(cf.\,\cite{NZ}*{Theorem~4})}\label{1conf}
Let  $\mathcal{M}$ be an $(h,\tau)$-compatible $\exp_{h,\tau}$-family.
Assume that  $\tau/\chi_{h,\tau}$ is constant on~$I$.
\begin{enumerate}
\setlength{\leftskip}{-15pt}
\item\label{561}
The function $\mathbb{I}_\tau$ is smooth and positive on $\mathcal{M}$.
\item\label{562}
Define a Riemannian metric $\overline{g}$ and a connection $\overline{\nabla}$ on $\mathcal{M}$
respectively by 
\[
\overline{g}\coloneqq \frac1{\mathbb{I}_{\tau}} \cdot g^{h,\tau},
\qquad
\overline{g}(\overline{\nabla}_V W,Z)=
\overline{g}(\nabla^{h,\tau}_V W,Z)-Z\left(\log \frac1{\mathbb{I}_{\tau}}\right) \overline{g}(V,W).
\]
Assume~\eqref{eexchange}.
Then the pair $(\overline{g},\overline{\nabla})$ is a Hessian structure on $\mathcal{M}$,
where any global $\exp_{h,\tau}$-normalization of $\mathcal{M}$ becomes a global potential.
\item\label{563}
The $(h,\tau)$-divergence is a contrast function on~$\mathcal{M}$
and the canonical divergence $\overline{D}$ of $(\overline{g},\overline{\nabla})$ on~$\mathcal{M}$ is given by 
\[
\overline{D}(p,p')=\frac{1}{\mathbb{I}_{\tau}(p)}D_{h,\tau}(p,p').
\]
\end{enumerate}
\end{proposition}
\begin{proof}
Let $(\mathcal{M}, \theta, T, c,\psi)$ be an $\exp_{h,\tau}$-coordinate representation on $\mathcal{M}$.
Set $\lambda\coloneqq \tau/\chi_{h,\tau}$.
Then we have $\chi'_{h,\tau}>0, m_{h,\tau}=\lambda \chi'_{h,\tau}/\chi_{h,\tau}$ on $I$ and 
observe from \eqref{C4} with \eqref{diff} that 
\begin{align*}
g^{h,\tau}_p(\partial_i, \partial_j)
&=\int_{{\mathcal{X}}} \partial_i e_x |_p\cdot \partial_je_x |_p \cdot m_{h,\tau}(p(x))d\mu(x)
=\partial_i \partial_j \psi |_p \cdot \mathbb{I}_{\tau}(p).
\end{align*}
This ensures that $\mathbb{I}_\tau$ is smooth and positive on $\mathcal{M}$, that is, Claim~\eqref{561} follows.

Let us prove Claim~\eqref{562}.
It follows from \eqref{flat} that 
\begin{align*}
\overline{g}(\overline{\nabla}_{\partial_i} \partial_j,\partial_k)
&=\frac1{\mathbb{I}_{\tau}} \cdot g^{h,\tau}(\nabla^{h,\tau}_{\partial_i} \partial_j, \partial_k)
+\frac{\partial_k \mathbb{I}_{\tau}}{\mathbb{I}_\tau^2}\cdot g^{h,\tau}(\partial_i,\partial_j)\\
&=- \frac1{\mathbb{I}_{\tau}}\cdot(\partial_i\partial_j\psi) \cdot \partial_k \mathbb{I}_\tau +\frac{ \partial_k \mathbb{I}_\tau}{\mathbb{I}_\tau^2}\cdot \partial_i \partial_j \psi \cdot \mathbb{I}_{\tau}(p)\\
&=0,
\end{align*}
which implies that $(\mathcal{M},\theta)$ is affine with respect to $\overline{\nabla}$ and 
\[
\overline{\nabla} d \psi (\partial_i, \partial_j)
=\partial_i \partial_j \psi=\frac1{\mathbb{I}_{\tau}} \cdot g^{h,\tau}(\partial_i,\partial_j)=\overline{g}(\partial_i, \partial_j).
\]
Thus $(\overline{g},\overline{\nabla})$ is a Hessian structure on $\mathcal{M}$ with the global potential $\psi$ and  Claim~\eqref{562} follows.

To prove Claim~\eqref{563},
define $\eta\colon \mathcal{M} \to \mathbb{R}^n$ by
$\eta(p)\coloneqq (\partial_i \psi |_p)_{i=1}^{n}$.
Then $\eta=\mathbb{I}_{T\tau }/\mathbb{I}_\tau$ follows from~\eqref{derideri}
and  the canonical divergence $\overline{D}$ of $(\overline{g},\overline{\nabla})$ on $\mathcal{M}$ is given by
\begin{align*}
\overline{D}(p,p')
&=\psi(p')-\psi(p)+\langle \theta (p)-\theta (p'), \eta(p)\rangle\\
&=\frac{1}{\mathbb{I}_\tau(p)} 
\int_{{\mathcal{X}}} \left( \psi(p')-\psi(p)+\langle \theta (p)-\theta (p'),T\rangle \right) \tau(p) d\mu\\
&=\frac{1}{\mathbb{I}_\tau(p)} 
\int_{{\mathcal{X}}} 
\left( \ell_{h,\tau}(p)-\ell_{h,\tau}(p') \right) \tau(p)d\mu,
\end{align*}
where we used~\eqref{aaa} in the last equality.
For $s,t \in I$, we calculate 
\begin{align*}
\left( \ell_{h,\tau}(t)-\ell_{h,\tau}(s) \right) \tau(t)
&=d_{h,\tau}(t,s)-s^\star_{h,\tau}(t)+s^\star_{h,\tau}(s),
\end{align*}
which yields
\begin{align*}
\overline{D}(p,p')
&=\frac{1}{\mathbb{I}_\tau(p)} 
\int_{{\mathcal{X}}} 
\left( d_{h,\tau}(p,p') -s^\star_{h,\tau}(p)+s^\star_{h,\tau}(p')\right)
d\mu.
\end{align*}
Since we have $s^\star_{h,\tau}{}'=-\tau\ell_{h,\tau}'=-\lambda$,
it turns out that 
\[
s^\star_{h,\tau}(s)-s^\star_{h,\tau}(t)=-\lambda(s-t)
\quad \text{for }s,t\in I,
\]
implying 
\begin{align*}
\overline{D}(p,p')
=\frac{1}{\mathbb{I}_\tau(p)} 
\int_{{\mathcal{X}}} 
 d_{h,\tau}(p,p')
 d\mu
=\frac{1}{\mathbb{I}_\tau(p)} D_{h,\tau}(p,p').
\end{align*}
Thus $D_{h,\tau}$ is a contrast function on $\mathcal{M}$
and
the proof of Claim~\eqref{563} is complete.
\end{proof}

\begin{remark}\label{ambiguity}
\begin{enumerate}
\setlength{\leftmargin}{23pt}
\setlength{\leftskip}{-15pt}
\item
Since $\mathbb{I}_{\mathrm{id}_I}=\mathbb{I}$ is identically $1$ on $\mathcal{P}_I(\mu)$,
as an immediate consequence of Theorem~\ref{mainthm22},
if the integration and the first and second differentials of $s^{\ast}_{h,\mathrm{id}_I} \circ e_x$  commute, that is, \eqref{exchange1},\eqref{exchange2} hold,
then the $(h, \mathrm{id}_I)$-divergence induces a Hessian structure on an $(h,\mathrm{id}_I)$-compatible $(h, \mathrm{id}_I)$-exponential family.
\item
There exists an example that satisfies the assumptions in Theorem~\ref{mainthm22} with $\tau\neq\mathrm{id}_I$.
Indeed, for $\lambda>0$, we define $(h_\lambda,\tau_\lambda,I)\in \mathcal{G}$ by
\[
h_\lambda(r)=\lambda^{-1}(r \log r-r), \qquad \tau_\lambda(t)=t^\lambda, \qquad I=(0,\infty),
\]
then $\exp_{h_\lambda, \tau_\lambda}(u)=e^u$ for $u\in \mathbb{R}$ 
but $(h_\lambda, \tau_\lambda,I)\not\simeq (h_{\lambda'}, \tau_{\lambda'},I)$ if $\lambda \neq \lambda'$.
For $\sigma>0$, set
\[
 \mathcal{M}_\sigma\coloneqq \left\{ p^{v}\coloneqq (2\pi\sigma^2)^{-\frac{d}{2}}\exp\left({-\frac{|\cdot-v|^2}{2\sigma^2}}\right) \biggm| v\in \mathbb{R}^d \right\}\subset \mathcal{P}(\mathcal{L}^d).
\]
Then $\mathcal{M}_\sigma$ is an $(h_\lambda, \tau_\lambda)$-compatible $\exp_{h_\lambda, \tau_\lambda}$-family 
and its $\exp_{h_\lambda, \tau_\lambda}$-coordinate representation $(\mathcal{M}_\sigma, \theta, \psi, T,c)$ is give by 
\[
\theta(p^{v})\coloneqq v, \qquad
T(x)=\frac{x}{\sigma^2},\qquad
c(x)=\frac{|x|^2}{2\sigma^2},\qquad
\psi(p^v)=\frac{|v|^2}{2\sigma^2}+\frac{d}{2} \log (2\pi\sigma^2).
\]
We find that 
\[
\mathbb{I}_{\tau_\lambda}(p^v)=\int_{\mathbb{R}^d} \tau_\lambda(p^{v}(x))dx=(2\pi\sigma^2)^{\frac{d}{2}(1-\lambda)} \lambda^{-\frac{d}{2}}
\]
and   the assumptions in Theorem~\ref{mainthm22} are valid for any $\lambda>0$.
\item\label{gfm}
While a contrast function satisfying Condition~\ref{Eguchi} determines a Riemannian metric, 
this correspondence is not injective.
Indeed, Barndorff-Nielsen and Jupp~\cite{BJ}*{Theorem~4.1} proved  three contrast functions satisfying Condition~\ref{Eguchi} induce the same Riemannian metric.
In our context, 
we have $g^{h,\tau}=g^{\tilde{h}, \tilde{\tau}}$ even if $(h,\tau, I)\not\simeq (\tilde{h}, \tilde{\tau},I)$
since the Riemannian metric $g^{h,\tau}$  is determined only by $m_{h,\tau}$, not by $(h,\tau)$  (see also in Remark~\ref{underinv}).
Naudts and Zhang \cite{NZ}*{Section~3} called this phenomenon the \emph{gauge freedom} of a Riemannian metric.
\end{enumerate}
\end{remark}

\section{Pythagorean relation}\label{6}
In order to establish the Pythagorean relation for the $(h,\tau)$-divergence, 
we introduce the \emph{$(h,\tau)$-entropy}. 
As well as a divergence, an entropy is often used but its definition is vague in information geometry.

In this paper, we regard an entropy as a generalization of an internal energy, that is, we consider only a trace-form entropy.
To be physical,
the density $\phi\in C^{\infty}([0,\infty))$ of the internal energy $\mathbb{I}_\phi$ on $\mathcal{P}(\mu)$ is assumed to vanish at $0$ and to be convex on $[0,\infty)$
since  the energy of no matter is zero and the thermodynamical pressure of $\phi$ given by
\[
t\mapsto t \phi'(t)-\phi(t)
\]
is nonnegative and nondecreasing
(see for example~\cite{Vi}*{\S5.2.3}).
The choice $\phi(r)=r\log r$ recovers the \emph{Boltzmann--Gibbs--Shannon entropy}, which is defined for $p\in \mathcal{P}(\mu)$ with $p \log p \in L^1(\mu)$ by
\[
S_{\mathrm{BGS}}(p)\coloneqq -\int_{{\mathcal{X}}} p(x) \log p(x) d\mu(x).
\]
However, in this paper, we do not assume that $\phi$ vanishes at $0$ and is convex on $[0,\infty)$.
\begin{definition}\label{ent}
For $(h,\tau, I)\in \mathcal{G}$,
we define  $s_{h,\tau}\in C^{\infty}(I)$ by
\[
s_{h,\tau}\coloneqq -h\circ \tau
\]
and we call $\mathbb{I}_{s_{h,\tau}}$ the \emph{$(h,\tau)$-entropy}. 
\end{definition}
\begin{remark}
Let $(h,\tau, I) \in \mathcal{G}$.
The thermodynamical pressure of $s_{h,\tau}$ is given by 
\[
t \mapsto t s'_{h,\tau}(t)-s_{h,\tau}(t)=s^\star_{h,\tau}(t)+(\tau(t)-t\tau'(t))\ell_{h,\tau}(t),
\]
hence $s^\star_{h,\tau}$ becomes the thermodynamical pressure of $s_{h,\tau}$ if $\tau=\mathrm{id}_I$.
Thus we may possibly regard $s^\star_{h,\tau}$ as a generalization of the thermodynamical pressure of $s_{h,\tau}$.
Note that, for $(h|_{\tau(I)},\tau, I)\in \mathcal{G}_\ast$,  
we have $s^\star_{h|_{\tau(I)},\tau}=s_{h^\ast|_{\ell_{h,\tau}(I)},\ell_{h,\tau}}$. 
\end{remark}
We apply the Pythagorean relation for the $(h,\tau)$-divergence to  maximize the $(h,\tau)$-entropy 
as similar as the Kullback--Leibler divergence and for  a maximization of the Boltzmann--Gibbs--Shannon entropy.
In~\cite{NZ}*{Section~5.4}, a maximization of the $(h,\tau)$-entropy is formally discussed by use of the method of Lagrange multipliers,
where the constraints are given in terms of \emph{escort distributions}.
We refer to \cite{Amari}*{Section 4.3} for the usage of escort distributions.
\begin{proposition}\label{maxcor}
For $(h,\tau,I)\in \mathcal{G}$, 
let 
$\mathcal{P}$ be a set in $\mathcal{P}_I(\mu)$ equipped with 
an $\exp_{h,\tau}$-representation  $(\mathcal{P}, \theta, T, c,\psi)$.
Assume that 
\[
\mathbb{I}_{\tau}, \mathbb{I}_{s_{h,\tau}}\mathbb{I}_{T\tau} \in \mathbb{R}^{\mathcal{P}}.
\]
Moreover, assume $ \mathbb{I}_{c\tau}\in \mathbb{R}^{\mathcal{P}} $  if $(h,\tau)\neq (r\log r, \mathrm{id}_I)$.
For $p\in \mathcal{P}_I(\mu)$, if
\[
\mathbb{I}_{\tau}, \mathbb{I}_{s_{h,\tau}},\mathbb{I}_{c\tau}, \mathbb{I}_{T\tau} \in \mathbb{R}^{\{p\}}
\]
and there exists $p_\ast\in \mathcal{P}$ such that 
\begin{equation}\label{constr}
\mathbb{I}_{\tau}(p)=\mathbb{I}_{\tau}(p_\ast), \qquad
\mathbb{I}_{T \tau}(p)=\mathbb{I}_{T\tau}(p_\ast),
\end{equation}
then it follows that
\[
D_{h,\tau} (p,p')=D_{h,\tau} (p,p_\ast)+D_{h,\tau} (p_\ast,p')
\quad\text{for }p'\in \mathcal{P}.
\]
Moreover, if $c$ vanishes $\mu$-a.e., then $\mathbb{I}_{s_{h,\tau}}(p_\ast)\geq \mathbb{I}_{s_{h,\tau}}(p)$  with equality if and only if $p=p_\ast$ for $\mu$-a.e.
\end{proposition}
\begin{proof}
For $p'\in \mathcal{P}$, it turns out that
\begin{align*}
d_{h,\tau}(p,p')
=&-s_{h,r}(p)+s_{h,r}(p')-(\tau(p)-\tau(p'))\ell_{h,\tau}(p')\\
=&-\left[s_{h,r}(p)+\tau(p)\cdot\left( \langle \theta(p'), T\rangle-c-\psi(p')\right)\right]\\
&+\left[s_{h,r}(p')+\tau(p')\cdot\left( \langle \theta(p'), T\rangle-c-\psi(p')\right)\right]
\quad\text{for $\mu$-a.e.}
\end{align*}
In the case of  $(h,\tau)= (r\log r, \mathrm{id}_I)$, we have
\begin{align*}
d_{h,\tau}(p,p')
&=-\left[-p\log p+ p\cdot\left( \langle \theta(p'), T\rangle-c-\psi(p')\right)\right]\\
&=-\left[s_{h,r}(p)+\tau(p)\cdot\left( \langle \theta(p'), T\rangle-c-\psi(p')\right)\right]
\quad\text{for $\mu$-a.e.}
\end{align*}
Then, by the assumption,  we have $d_{h,\tau}(p,p')\in L^1(\mu^{\otimes 2}) $ hence 
\begin{align*}
D_{h,\tau}(p,p')
=&-\left(\mathbb{I}_{s_{h,r}}(p)+\langle \theta(p'), \mathbb{I}_{T \tau}(p)\rangle-\mathbb{I}_{c\tau}(p)-\psi(p')\mathbb{I}_\tau(p)\right)\\
&+\left(\mathbb{I}_{s_{h,r}}(p')+ \langle  \theta(p'),\mathbb{I}_{T\tau}(p')\rangle-\mathbb{I}_{c\tau}(p')-\psi(p')\mathbb{I}_\tau(p')\right)
\quad\text{if $(h,\tau)\neq (r\log r, \mathrm{id}_I)$}
\end{align*}
and 
\[
D_{h,\tau}(p,p')
=-\left(\mathbb{I}_{s_{h,r}}(p)+\langle \theta(p'), \mathbb{I}_{T \tau}(p)\rangle-\mathbb{I}_{c\tau}(p)-\psi(p')\mathbb{I}_\tau(p)\right)
\quad\text{if $(h,\tau)= (r\log r, \mathrm{id}_I)$}.
\]
Similarly, we have $d_{h,\tau}(p_\ast,p')\in L^1(\mu^{\otimes 2}) $ and
 observe from~\eqref{constr} that
\begin{align*}
D_{h,\tau} (p,p_\ast)+D_{h,\tau} (p_\ast,p')=D_{h,\tau} (p,p').
\end{align*}
Moreover, if $c$ vanishes $\mu$-a.e., then
\begin{align*}
0\leq D_{h,\tau}(p,p_\ast)=-\mathbb{I}_{s_{h,r}}(p)+\mathbb{I}_{s_{h,r}}(p_\ast),
\end{align*}
implying  $\mathbb{I}_{s_{h,\tau}}(p_\ast) \geq \mathbb{I}_{s_{h,\tau}}(p_\ast)$ 
with equality if and only if $D_{h,\tau}(p,p_\ast)=0$, that is, $p=p_\ast$ for $\mu$-a.e.
Thus the proof is archived.
\end{proof}
\begin{remark}\label{standard}
\begin{enumerate}
\setlength{\leftmargin}{23pt}
\setlength{\leftskip}{-15pt}
\item
For $(h|_{\tau(I)},\mathrm{id}_I, I)\in \mathcal{G}_\ast$ with $J_{h^\ast}=\mathbb{R}$,
the $(h|_{\tau(I)},\mathrm{id}_I)$-divergence is called the \emph{$U$-divergence} 
and the Pythagorean relation for the $U$-divergence is proved in~\cite{MTKE}*{Theorem~1}.
This Pythagorean relation is irrelevant to the flatness of the connection $\nabla^{h|_{\tau(I)},\tau}$ and the flatness comes from the Euclidean connection 
even though a $U$-flat subspace can be regarded a counterpart of an $(h|_{\tau(I)},\mathrm{id}_I)$-exponential family
(see \cite{MTKE}*{Definition~2}).
See also \cite{NZ}*{Theorems~\text{1--4}}, where the Euclidean connection on a deformed exponential family is used.
\setlength{\itemindent}{0pt}
\item
\setlength{\leftskip}{-15pt}
For an $\exp_{h,\tau}$-representation $(U,\theta, T, c,\psi)$ on a set $\mathcal{P}$ in $\mathcal{P}(\mu)$,
the notion that $c$ vanishes $\mu$-a.e. corresponds to the notion for $(U,\theta, T, c,\psi)$ being \emph{standard} in~\cite{BN1970}*{Chapter~5}.
In the case of $(h,\tau,I)=(r\log r, \mathrm{id}_{(0,\infty)}, (0,\infty))$, 
we can always use an standard  exponential representation
since if we fix $p_0\in U$ and set 
\begin{align*}
\mathcal{P}_0\coloneqq \{pp_0^{-1}\mid p\in \mathcal{P}\},\qquad
U_0\coloneqq \{pp_0^{-1}\mid p\in U \},
\end{align*}
then $(U_0, \theta-\theta(p_0), T, 0, \psi-\psi(p_0))$ is an exponential representation on  $\mathcal{P}_0$ in $\mathcal{P}(p_0\mu)$.
\end{enumerate}
\end{remark}
\section{Proof of Theorem~\ref{mainthm3}}\label{proof3}
In order to infer the law of an unknown ${\mathcal{X}}$-valued random variable on a probability space $(\Omega, \mathbb{P})$,
we repeat trials and observe data $\bm{x}_\ast \in {\mathcal{X}}^k$, where $k\in \mathbb{N}$.
This corresponds to observing the values $(X_m(\omega))_{m=1}^k$ at a sample point $\omega\in \Omega$ of identically distributed ${\mathcal{X}}$-valued random variables $(X_m)_{m=1}^k$.
The independence of $(X_m)_{m=1}^k$ is often assumed
and 
the \emph{maximum likelihood estimation} is a method of statistical inference using i.i.d.\,random variables.
To be precise, for given observed data $\bm{x}_\ast \in {\mathcal{X}}^k$,
the law of an unknown ${\mathcal{X}}$-valued random variable  is inferred as a maximizer of the function given by
\[
p\mapsto e_{\bm{x}_\ast}(p^{\otimes k})=p^{\otimes k}(\bm{x}_\ast)
=\prod_{m=1}^k p(\pi_m^k(\bm{x}_\ast))
\quad\text{on a subset $\mathcal{P}$ in $\mathcal{P}(\mu)$},
\]
where 
$\pi_m^k\colon \mathcal{X}^k \to \mathcal{X}$ denotes the projection onto the $m$th coordinate.
Existence and uniqueness of maximizers for the above function are not necessarily guaranteed.

To see the validity of the maximum likelihood estimation, we consider an independent toss of an unfair coin as a typical and simple example. 
Let 
${\mathcal{X}}\coloneqq \{{\rm head}, {\rm tail}\}$
and $\mu$ be a counting measure.
Then $\mathcal{P}(\mu)$ is an exponential family, where an exponential coordinate representation $(\mathcal{P}(\mu),\theta, T,c, \psi)$ is given by 
\begin{align*}
\theta(p)&=\log p(\text{head})-\log p(\text{tail}),
&&\psi(p)=-\log p(\text{tail}),\qquad
 \text{for }p\in \mathcal{P}(\mu), \\
T(\text{head})&=1, \qquad
T(\text{tail})=0,
&&c(\text{head})=c(\text{tail})=0.
\end{align*}
Given $\bm{x}_\ast\in \mathcal{X}^k$ with $\#\{ m\mid \pi_m^k(\bm{x}_\ast)={\rm head} \}\neq 0,k$, 
it is natural to infer the probability density $p_k$ associated to an unfair coin toss from the observed data $\bm{x}_\ast$ as 
\[
p_k(\text{head})=\frac{\#\{ m \mid \pi_m^k(\bm{x}_\ast)={\rm head} \} }{k},
\]
which maximizes the function $p\mapsto e_{\bm{x}_\ast}(p^{\otimes k})$ on $\mathcal{P}(\mu)$.
Thus  the maximum likelihood estimation naturally appears.
Repeating coin tosses and observing more data will improve the accuracy of the estimate $p_k$. 
Indeed, if $p_\ast\in \mathcal{P}(\mu)$ describes the distribution of the unfair coin 
and we choose i.i.d.\,random variables $(X_m)_{m=1}^k$ with distribution $p_\ast\mu$,
then the strong law of large numbers gives
\begin{align*}
\frac1k \sum_{m=1}^k (T\circ X_m)
\xrightarrow{k\to\infty}
\int_\Omega T(X(\omega))d\mathbb{P}(\omega)
=\mathbb{I}_T(p_\ast)
\quad\text{$\mathbb{P}$-a.s.,}
\end{align*}
that is, 
the average of heads in $k$ coin tosses converges to the probability of heads in the coin toss 
almost surely as $k\to \infty$.
Since $\mathbb{I}_T\colon \mathcal{P}(\mu)\to (0,1)$ is bijective, an infinite number of coin tosses tells the true distribution of the unfair coin.

This fact is extended to  an exponential family as follows
(see \cite{Amari}*{Chapter~2.8.3} for a formal discussion).
\begin{proposition}\label{convex}
For $\Phi\colon \mathcal{X}\to \mathbb{R}^n$ and $k\in \mathbb{N}$,
define $\Phi^{+k}\colon \mathcal{X}^k\to \mathbb{R}^n$~by
\[
\Phi^{+k}(\bm{x})\coloneqq\sum_{m=1}^k \Phi(\pi_m^k(\bm{x})).
\]

Let $\mathcal{M}$ be an exponential family in $\mathcal{P}(\mu)$ equipped with an exponential coordinate representation $(\mathcal{M},\theta,T,c, \psi)$.
\begin{enumerate}
\setlength{\leftskip}{-15pt}
\item\label{712}

For ${\bm x_\ast}\in \mathcal{X}^k$ and $p_\ast \in \mathcal{M}$, 
if  $p_\ast$  maximizes  the function $p\mapsto e_{\bm{x}_\ast}(p^{\otimes k})$ on $\mathcal{M}$,
then $\mathbb{I}_{T^{+k}}(p_\ast^{\otimes k})=T^{+k}(\bm{x}_\ast)$ holds and 
$p_\ast$ is a unique minimizer~of 
\[
p\mapsto \KL (\rho,p^{\otimes k})\quad \text{on } \mathcal{M}
\]
for $\rho\in \mathcal{P}(\mathcal{X}^k)$ 
with 
\[
\mathbb{I}_{r\log r}, \mathbb{I}_{c^{+k}},\mathbb{I}_{T^{+k}}\in \mathbb{R}^{\{\rho\}} \text{ and }
\mathbb{I}_{T^{+k}}(\rho)=T^{+k}(\bm{x}_\ast).
\]
\item\label{713}
For $p\in \mathcal{M}$,
let $(X_k)_{k\in\mathbb{N}}$ be ${\mathcal{X}}$-valued i.i.d.\,random variables with distribution~$p\mu$.
Then 
\begin{align*}
\frac1k T^{+k}\circ (X_1,\ldots,X_k)
\xrightarrow{k\to\infty}\mathbb{I}_{T}(p)
\quad \text{$\mathbb{P}$-a.s.},
\end{align*}
and quantitatively, for $\varepsilon>0$,
\begin{align*}
\mathbb{P}\left( \left|\frac1k T^{+k}\circ (X_1,\ldots,X_k)-\mathbb{I}_{T}(p)\right| >\varepsilon \right)
&\leq \frac{1}{k\varepsilon^2} 
\int_{\mathcal{X}} |T(x)-\mathbb{I}_T(p)|^2 p(x)d\mu(x).
\end{align*}
\end{enumerate}
\end{proposition}
By Proposition~\ref{convex}\eqref{712}, the uniqueness of 
maximum likelihood estimators follows.
We observe from Proposition~\ref{convex}\eqref{713} that an infinite number of trials determines the true distribution 
since $(\mathcal{M},\mathbb{I}_T)$ is a coordinate system. 
Moreover, the quantitative error is expressed as the Fisher metic by Lemma~\ref{lem}\eqref{463}.
We omit the proof of the proposition
since  Claim~\eqref{712} is well-known and  Claim~\eqref{713} is exactly the strong law of large numbers and 
Chebyshev's inequality.
In addition, we generalize the proposition  to a deformed exponential family (see Theorem~\ref{mainthm33}) 
and the proof of the proposition is similar to that of the generalization.

For the generalization,
we introduce the identical repeatability of a deformed exponential family.
\begin{definition}[restatement of Definition~\ref{repeat}]
Let $(h,\tau,I)\in \mathcal{G}$. 
An $\exp_{h,\tau}$-family $\mathcal{M}$ is \emph{identically repeatable}
if, for each $k\in \mathbb{N}$,
there exist a positive and increasing smooth function $\alpha_k$ on $I$, a positive smooth function $b_k$ on $\mathcal{M}$ and an injection $\iota_k \colon \mathcal{M}\to \mathcal{P}_I(\mu^{\otimes k})$ 
such that $\iota_k(\mathcal{M})$ is a statistical model in $\mathcal{P}(\mu^{\otimes k})$ and,
for each global $\exp_{h,\tau}$-coordinate representation $(\mathcal{M},\theta,T,c,\psi)$ on $\mathcal{M}$,
there exists $\psi_k\in C^{\infty}(\mathcal{M})$ such that  
\[
\iota_k(p)({\bm x})
=\alpha_k
\left(
\exp_{h,\tau}\left(  b_k(p)  \left( \langle \theta(p), T^{+k}({\bm{x}}) \rangle-c^{+k}( {\bm{x}}) \right) -\psi_k(p) \right)
\right)
\]
for $p\in \mathcal{M}$ and ${\bm x} \in {\mathcal{X}}^k$, and furthermore
\begin{align}\label{marginal}
\int_{\mathcal{X}^{k+k'}}\iota_{k+k'}(p)({\bm x},{\bm y})d\mu^{\otimes k'}(\bm{ y})
=\iota_k(p)({\bm x})
\quad \text{for }k'\in \mathbb{N}.
\end{align}
We call the injection $(\iota_k)_{k\in \mathbb{N}}$ the \emph{repetition}.

For an identically repeatable $\exp_{h,\tau}$-family $\mathcal{M}$ equipped with a repetition $(\iota_k)_{k\in\mathbb{N}}$
and $p\in \mathcal{M}$, 
we say random variables $(X_k)_{k\in \mathbb{N}}$ 
are \emph{$(h,\tau)$-dependent and identically distributed according to~$p$}
if the law of $(X_m)_{m=1}^k$ is $\iota_k(p) \mu^{\otimes k}$ for each $k\in \mathbb{N}$.
\end{definition}
To verify the identical repeatability of an $\exp_{h,\tau}$-family~$\mathcal{M}$,
it is enough to consider one global $\exp_{h,\tau}$-coordinate representation $(\mathcal{M},\theta, T,c,\psi)$ rather than all of  them by Proposition~\ref{minimal}.

To give an identically repeatable deformed exponential family
excepted for exponential families,
we focus on the probability densities on  $\mathcal{X}=\mathbb{R}^d$ with mean-like and covariance-like parameters
and on the case that $\chi_{h,\tau}$ is given by a power function.
\begin{definition}
For $q>0$, define the three smooth functions $\ln_q, h_q, f_q\colon (0,\infty)\to \mathbb{R}$ by 
\[
\ln_q(r)\coloneqq \int_{1}^r t^{-q}dt,\quad
h_q(r)\coloneqq \int_{1}^r \ln_q(t)dt, \quad
f_q(r)\coloneqq qr \ln_{q}(r^{\frac1{q}})-r.
\]
Set 
\[
\chi_q\coloneqq \chi_{(h_q, \mathrm{id}_{(0,\infty)})},
\qquad
\exp_q\coloneqq \exp_{(h_q, \mathrm{id}_{(0,\infty)})}.
\]
\end{definition}
We easily find $\chi_q(t)=t^q$ for $t\in (0,\infty)$ and  $(h_q, \mathrm{id}_{(0,\infty)}, (0,\infty)), (f_q,\chi_{q}, (0,\infty))\in \mathcal{G}$.
Moreover, we have
\[
\ell_{f_q,\chi_{q}}=\ell_{h_q,\mathrm{id}_{(0,\infty)}}=\ln_q,
\]
implying $\exp_{f_q,\chi_q}=\exp_q$.

Let us  consider an $\exp_q$-family in $\mathcal{P}(\mathcal{L}^d)$ 
with mean-like and covariance-like parameters such as 
\[
p(x)=\exp_q\left(-\langle x-v, S (x-v)\rangle-\lambda   \right)
\quad \text{for }x\in \mathbb{R}^d,
\]
where $v\in \mathbb{R}^d$, $S$ is a positive definite symmetric matrices of size~$d$ 
and $\lambda\in \mathbb{R}$ is a normalization.
We see that  the support of such $p$ is $\mathbb{R}^d$ if and only if  $q\geq 1$.
For $q=1$, we have 
$\ln_1(r)=\log (r)$ for $r>0$ and
$\exp_1(u)=e^{u}$ for $u\in \mathbb{R}$.
If  $q>1$
then 
\[
\ln_q(r)\coloneqq \frac{r^{1-q}-1}{1-q} \quad \text{for }r>0, \qquad
\exp_q(u)\coloneqq \begin{cases} 
\infty &\text{if } u \geq \dfrac{1}{q-1}. \\
[1+(1-q)u ]^{\frac{1}{1-q}} & \text{if } u<\dfrac{1}{q-1}.
\end{cases}
\]
\begin{definition}
We denote by $\mathrm{Sym}^+(d,\mathbb{R})$ the space of positive definite symmetric matrices of size~$d$.
For $S\in \mathrm{Sym}^+(d,\mathbb{R})$ and $x\in \mathbb{R}^d$, set 
\[
|x|^2_S\coloneqq \langle x, S x\rangle.
\]
For $1\leq i \leq d$,
we denote by $e_i$ the $d$-tuple consisting of zeros except for a 1 in the $i$th spot.
Define $F_i,F_{ij}\colon \mathbb{R}^d\to \mathbb{R}$ respectively by
\begin{align*}
F_i(x)&\coloneqq \langle x,e_i\rangle,\qquad
F_{ij}(x)\coloneqq\langle x,e_i\rangle\langle x,e_j\rangle= F_i(x)F_j(x),
\quad\text{for }x\in \mathbb{R}^d.
\end{align*}
\end{definition}
Let $v\in \mathbb{R}^d$ and $S\in \mathrm{Sym}^+(d,\mathbb{R})$.
For $q\geq 1,a>0$ and $l\geq 0$, 
we easily see that  there exists $\lambda\in \mathbb{R}$ such that
\[
\exp_q\left(-|x-v|^2_S-\lambda   \right)^a
\]
is contained in $\mathcal{P}(\mathcal{L}^d)$  and has a finite $l$ th moment 
if and only if 
\[
 (d+l)(q-1)<2a.
\]
For $a>0$ and $\lambda\in \mathbb{R}$, it tuns out that 
\[
\int_{\mathbb{R}^d} \exp\left(-|x-v|_S^2-\lambda   \right)^a dx
=e^{-a\lambda} \left(\det \frac{aS}{\pi}\right)^{-\frac{1}{2}}.
\]
For $q>1, a>0$ with  $d(q-1)<2a$ and  $\lambda>-(q-1)^{-1}$,
we calculate 
\begin{align*}
\int_{\mathbb{R}^d} \exp_q\left(-|x-v|_S^2-\lambda   \right)^a dx 
&=\int_{\mathbb{R}^d} [1+(q-1)( |x|^2_S+\lambda)]^{\frac{a}{1-q}} dx\\ 
&=\frac{ \det\left(\frac{q-1}{\pi}S\right)^{-\frac12} }{\Gamma(\frac{d}{2})}[1+(q-1)\lambda]^{\frac{a}{1-q}+\frac{d}{2}}\int_0^\infty (1+r)^{\frac{a}{1-q}} r^{\frac{d}{2}-1}dr\\
&= \det\left(\frac{q-1}{\pi}S\right)^{-\frac12} \frac{\Gamma(\frac{a}{q-1}-\frac{d}{2})}{\Gamma(\frac{a}{q-1})} [1+(q-1)\lambda]^{\frac{a}{1-q}+\frac{d}{2}},
\end{align*}
where we used the property 
\[
\int_0^\infty \frac{r^{u-1}}{(1+r)^{v}}dr=\frac{\Gamma(u)\Gamma(v-u)}{\Gamma(v)}\quad \text{for }u>0, v+u>0.
\]
This in turn justifies the following definition.
Throughout the rest of the paper, we always assume $q\geq 1$ and $d(q-1)<2$.
\begin{definition}
For $v\in \mathbb{R}^d$ and $S\in \mathrm{Sym}^+(d,\mathbb{R})$, define $p_{q}^{v,S}\in \mathcal{P}(\mathcal{L}^d)$ by
\[
p_{q}^{v,S}(x)\coloneqq \exp_q\left(- |x-v|_S^2-\lambda_{q}(S)\right) \quad \text{for }x\in \mathbb{R}^d,
\]
where 
\[
\lambda_{q}(S)\coloneqq 
-\ln_q\left(
\left[\det\left(\frac{q-1}{\pi}S\right)^{\frac{1}2} \frac{\Gamma(\frac{1}{q-1})}{\Gamma(\frac{1}{q-1}-\frac{d}{2})}\right]^{\frac{2}{2+d(1-q)}}
\right)
\]
if $q>1$ and $\lambda_{1}(S)\coloneqq -\log \det(S/\pi)^{1/2}$.
Set 
\begin{align*}
\mathcal{M}_{q,d}&\coloneqq \{p_{q}^{v,S}\mid v\in \mathbb{R}^d, S\in \mathrm{Sym}^+(d,\mathbb{R}) \},\\
\mathcal{M}_{q,d}^{I_d}&\coloneqq \{p_{q}^{v,I_d}\mid v\in \mathbb{R}^d \},\\
\mathcal{M}_{q,d}^{\tr}&\coloneqq \{ p_{q}^{v,S}\mid  v\in \mathbb{R}^d, S\in \mathrm{Sym}^+(d,\mathbb{R})  \text{ with }\tr S=d \}.
\end{align*}
\end{definition}
Since $\mathcal{M}_{q,d}^{I_d}=\mathcal{M}_{q,d}^{\tr}$ holds for $d=1$, 
we implicitly  assume $d\geq 2$ when we treat $\mathcal{M}_{q,d}^{\tr}$.
For $1\leq i,j \leq d$, 
define $\vartheta^i, \vartheta^{ij}\colon\mathcal{M}_{q,d}\to \mathbb{R}$ by
\begin{align*}
\vartheta^{i}(p_{q}^{v,S})&\coloneqq 2\langle e_i, Sv\rangle, \qquad
\vartheta^{ij}(p_{q}^{v,S})\coloneqq -\left(2-\delta_{ij}\right) \langle e_i, S e_j\rangle.
\end{align*}
We also define $\psi:\mathcal{M}_{q,d}\to \mathbb{R}$ by
\[
\psi(p_{q}^{v,S})\coloneqq |v|_S^2+\lambda_{q}(S).
\]
Then $\mathcal{M}_{q,d}, \mathcal{M}_{q,d}^{I_d}, \mathcal{M}_{q,d}^{\tr}$ are $\exp_q$-families,
where their  deformed exponential coordinate representation are respectively given by 
\begin{align}\notag 
&\left(\mathcal{M}_{q,d}, ( \vartheta^{i}, \vartheta^{ij})_{1\leq i \leq j \leq d }, ( F_{i}, F_{ij})_{1\leq i \leq j \leq d}, 0, \psi\right),\\ \label{qd}
&\left(\mathcal{M}_{q,d}^{I_d}, ( \vartheta^{i})_{1\leq i \leq d }, ( F_{i})_{1\leq i \leq d}, |\cdot|^2, \psi\right),\\ \notag
& \left(
\mathcal{M}_{q,d}^{\tr},
(\vartheta^{i}, \vartheta^{ij})_{1\leq i\leq j \leq d  \text{ and } (i,j) \neq (d,d)}, 
( F_{i}, F_{ij}-\delta_{ij}F_{dd})_{1\leq i\leq j \leq d  \text{ and } (i,j) \neq (d,d)}, d\langle \cdot,e_d\rangle^2, \psi\right).
\end{align}

We will show the identical repeatability of $\mathcal{M}_{q,d}^{\tr}$ and $\mathcal{M}_{q,d}^{I_d}$.
\begin{definition}
For $S\in \mathrm{Sym}^+(d,\mathbb{R})$
and $k\in \mathbb{N}\cup\{0\}$,
we define 
\begin{align*}
a_{k}&\coloneqq 1+ (k+3) \cdot \frac{d(q-1)}{2},
\qquad q_{k}\coloneqq 1+\frac{q-1}{a_{k}},\\
\beta_{k}(S)&\coloneqq
\begin{dcases}
 \det\left(\frac{q-1}{\pi}S\right)^{-\frac{1}{d}} \Gamma\left(\frac{1}{q_k-1}\right)^{1-q_k} 
& \text{if $q>1$},\\
  \det(S)^{-\frac1d} &\text{if $q=1$},
\end{dcases}   \\
\nu_{k}
&\coloneqq 
\begin{dcases}
-\ln_q
\left(\Gamma\left(\frac{1}{q_k-1}\right)^{\frac1{a_{k}}}\Gamma\left(\frac{1}{q_0-1}\right)^{-\frac{1}{a_{0}}}\right)
& \text{if $q>1$},\\
\frac{dk}{2}\log \pi&\text{if }q=1.
\end{dcases}
\end{align*}
In the case of $k\in \mathbb{N}$, 
we define
\[
\rho_{q,k}^{v,S}({\bm x})
=\exp_{q}\left(-\sum_{m=1}^k | \pi^k_m(\bm{x})-v|^2_{\beta_{k}(S)S} -\nu_{k}\right)^{a_{k}}
\quad\text{for $v\in \mathbb{R}^d$ and $\bm{x}\in (\mathbb{R}^d)^k$}.
\]
\end{definition}
A direct calculation yields $\rho_{q,k}^{v,S}\in \mathcal{P}(\mathcal{L}^{dk})$.
Note that
\[
 \rho_{q,k}^{v,S}=\rho_{q,k}^{v,sS} \quad \text{for $s>0$}\quad\text{and}\quad  
\det (\beta_{k}(S)S)=\det (\beta_{k}(I_d)I_d).
\]
Moreover, 
\begin{align*}
\{\rho_{q,k}^{v,S} \mid v\in \mathbb{R}, S\in \mathrm{Sym}^+(d,\mathbb{R}) \}
=\{\rho_{q,k}^{v,S} \mid v\in \mathbb{R}, \mathrm{Sym}^+(d,\mathbb{R}), \mathrm{tr}S=d \},\quad
\{\rho_{q,k}^{v,I_d} \mid v\in \mathbb{R} \}
\end{align*}
are statistical models in $\mathcal{P}(\mathcal{L}^{dk})$
since they are submanifolds of the statistical model~$\mathcal{M}_{q_k, dk}$
(see Remark~\ref{4th}\eqref{fourth}),
where Condition~\eqref{C4} holds by the dominated convergence theorem.
\begin{remark}\label{4th}
\begin{enumerate}
\setlength{\leftmargin}{23pt}
\setlength{\leftskip}{-15pt}
\item
For $v\in \mathbb{R}^d$, 
$p_{q}^{v,I_d}= \rho_{q,1}^{v,I_d}$ holds if and only if $q=1$.
\item
\setlength{\leftskip}{-15pt}
We can define an injection  $\iota_k$ from
\[
\mathcal{M}_{q,d}^{\det}\coloneqq \{ p_{q}^{v,S}\in \mathcal{M}_{q,d}\mid \det S=1 \}
\]
to $\mathcal{P}(\mathcal{L}^{dk})$ sending $p_q^{v,S}$ to $\rho_{q,k}^{v,S}$.
However, we prefer $\mathcal{M}_{q,d}^{\tr }$ to $\mathcal{M}_{q,d}^{\det}$
since $\mathcal{M}_{q,d}^{\det}$ is a deformed exponential family if and only if $d=1$.
\item\label{fourth}
We find  that $\iota_{k}(\mathcal{M}_{q,d})$ is a submanifold of $\mathcal{M}_{q_k,dk}$ since we have
\[
\exp_{q_k}(u)=
[1+(1-q_k) u]^{\frac{1}{1-q_k}}
=\exp_{q}(a_{k}^{-1}u)^{a_{k}}
\quad \text{for }u<\frac{1}{q_k-1}=\frac{a_{k}}{q-1}.
\]
We observe from
\begin{align*}
(dk+4)(q_k-1)
=\frac{(dk+4)(q-1)}{a_{k}}
<\frac{2+(k+3)d(q-1)}{a_{k}}=2,
\end{align*}
that the fourth moment of any element in $\mathcal{M}_{q_k,dk}$ is finite
and $\iota_{k}(\mathcal{M}_{q,d})$ becomes a statistical model.
Similarly, we derive from $(dk+2)(q_k-1)<2q_k$ that 
\[
\quad \mathbb{I}_{\chi_{q_k}},
\mathbb{I}_{s_{f_{q_k},\chi_{q_k}}},
\mathbb{I}_{(F_i\circ \pi_m^k)\chi_{q_k}},
\mathbb{I}_{(F_{ij}\circ \pi_m^k)\chi_{q_k}} \in \mathbb{R}^{\mathcal{M}_{q_k,dk}}
 \quad\text{for }1\leq i,j\leq d \text{ and }1\leq m\leq k.
\]
\end{enumerate}
\end{remark}
\begin{proposition}
Both  $\mathcal{M}_{q,d}^{\tr}$ and $\mathcal{M}_{q,d}^{I_d}$ are  identically repeatable $\exp_q$-families.
\end{proposition}
\begin{proof}
For $\mathcal{M}=\mathcal{M}_{q,d}^{\tr}, \mathcal{M}_{q,d}^{I_d}$,
let $(\mathcal{M}, \theta, T, c,\psi)$ be its $\exp_q$-coordinate representation 
given in \eqref{qd}.
If we define 
\begin{align*}
\alpha_{k}(r) &\coloneqq r^{a_{k}},\qquad
b_{k}(p_q^{v,S}) \coloneqq  \beta_{k}(S),\qquad
\psi_{k}(p_{q}^{v,S})\coloneqq k|v|_{\beta_{k}(S)S}^2+\nu_{k},\\
\iota_{k}(p_q^{v,S})
&\coloneqq
\alpha_{k}
\left(
\exp_{h,\tau}\left(  b_{k}(p_q^{v,S}) \left( \langle \theta(p_q^{v,S}), T^{+k}( \bm{x})\rangle-c^{+k}( \bm{x}) \right) -\psi_{k}(p_q^{v,S}) \right)
\right),
\end{align*}
then $\iota_{k}(p_q^{v,S})=\rho_{q,k}^{v,S}$.
We easily see  $\iota_{k}$ is injective on $\mathcal{M}$.

The rest is to prove the marginal condition~\eqref{marginal}, that is, 
\[
\int_{(\mathbb{R}^d)^{k'}}
\rho_{q,k+k'}^{v,S}(\bm{x},\bm{y})d\bm{y}=\rho_{q,k}^{v,S}(\bm{x})
\quad\text{for }k,k'\in \mathbb{N},\bm{x}\in (\mathbb{R}^{d})^{k},  v\in \mathbb{R}^d \text{ and }S\in \mathrm{Sym}^+(d,\mathbb{R}).
\]
Since the case of $q=1$ is trivial, we assume $q>1$.
Fix $\bm{x}\in (\mathbb{R}^d)^{k}$ and set 
\begin{align*}
B^{v,S}(\bm{x})
&\coloneqq1+(q-1)\nu_{k+k'}+(q-1)\sum_{m=1}^{k}|\pi_{k'}^k(\bm{x})-v|^2_{\beta_{k+k'}(S)S}\\
&=\frac{\beta_{k+k'}(S)}{\beta_{k}(S)}\left[1+(q-1)\nu_{k}+(q-1)\sum_{m=1}^{k}|\pi^k_{m}(\bm{x})-v|_{\beta_{k}(S)S}\right].
 \end{align*}
Note that $\beta_{k+k'}(S)/\beta_{k}(S)$ does not depend on $S$.
A direct calculation provides
\begin{align*}
& \int_{(\mathbb{R}^d)^k}\rho_{q,k+k'}^{v,S}(\bm{x},\bm{y})d\bm{y}\\
&=B^{v,S}(\bm{x})^{\frac{1}{1-q_{k+k'}}}
\det\left(\frac{\beta_{k+k'}(S)}{\pi}S\right)^{-\frac{k'}{2}}
 \frac{1}{\Gamma(\frac{dk'}{2})}
\cdot \left(\frac{B(\bm{x})}{q-1}\right)^{\frac{dk'}{2}}
\int_0^\infty (1+t)^{\frac{1}{1-q_{k+k'}}}t^{\frac{dk'}{2}-1}dr\\
&=B^{v,S}(\bm{x})^{\frac{a_{k}}{1-q}}
\det\left(\frac{(q-1) \beta_{k+k'}(S) }{\pi}S\right)^{-\frac{k'}{2}}
 \frac{\Gamma(\frac{1}{q_k-1})}{\Gamma(\frac{1}{q_{k+k'}-1})}\\
&=\left[ B^{v,S}(\bm{x})
\cdot \frac{\beta_{k}(S) } 
{\beta_{k+k'}(S) }\right]^{\frac{1}{1-q_k}} \\
&=\rho_{q,k}^{v,S}(\bm{x})
\end{align*}
as desired.
\end{proof}

Now we prove a counterpart of Proposition~\ref{convex}, that is, 
if the law of an unknown random variable lies in an exponential family,
then a maximizer of the likelihood function for given data $\bm{x}_\ast \in \mathcal{X}^k$ becomes a unique minimizer  of the dual problem in terms of the Kullback--Leibler divergence.
Moreover,
the law of $X$ is deduced from an infinite trials.
\begin{theorem}\label{mainthm33}
For $k\in \mathbb{N}$, set
\begin{align*}
N_k\coloneqq \frac{d(d+3)}{2}k,
\qquad
T\coloneqq (F_i\circ \pi^k_{m}, F_{ij}\circ  \pi^k_{m})_{1\leq i \leq j\leq d, 1\leq m\leq k}\colon\mathbb{R}^d\to \mathbb{R}^{N_k}.
\end{align*}
\begin{enumerate}
\setlength{\leftskip}{-15pt}
\item\label{791}
Fix ${\bm x_\ast}\in (\mathbb{R}^d)^k$ and  $p_\ast\in \mathcal{M}_{q,d}^{\tr }$. 
If $ \iota_{k}(p_\ast)$ maximizes the function  $e_{\bm{x}_\ast} $ on $\mathcal{M}_{q_k,dk}$,
then 
\begin{align}\label{thmeq1}
\mathbb{I}_{T\chi_{q_k}}(\iota_{k}(p_\ast))=\mathbb{I}_{\chi_{q_k}}(\rho_{q,k}^{0,I_d}) \cdot T(\bm{x}_\ast)
\end{align}
and 
$\iota_{k}(p_\ast)$ is a unique minimizer of $D_{f_{q_k}, \chi_{q_k}}(\rho,\cdot)$ on $\mathcal{M}_{q_k,dk}$ for any $\rho\in \mathcal{P}(\mathcal{L}^{dk})$ 
with 
\begin{align}\label{rhoo}
\begin{split}
&\mathbb{I}_{\chi_{q_k}}, \mathbb{I}_{s_{f_{q_k},\chi_{q_k}}}, \mathbb{I}_{T\chi_{q_k}} \in \mathbb{R}^{\{\rho\}} \text{ and }\\&\mathbb{I}_{\chi_{q_k}}(\rho)=\mathbb{I}_{\chi_{q_k}}(\rho_{q,k}^{0,I_d}),\qquad
\mathbb{I}_{T\chi_{q_k}}(\rho)=\mathbb{I}_{\chi_{q_k}}(\rho_{q,k}^{0,I_d})\cdot T(\bm{x}_\ast).
\end{split}
\end{align}
\item\label{792}
Let $(X_k)_{k\in\mathbb{N}}$ be $(h_q,\mathrm{id}_{(0,\infty)})$-dependent and  identically distributed random variables according to $p\in \mathcal{M}_{q,d}^{\tr}$.
For $1\leq i,j\leq d$ and  for $\varepsilon>0$,
\begin{align*}
&\mathbb{P}\left(\left |
\frac1k{F}_{i}^{+k}\circ (X_1,\ldots,X_k)-\mathbb{I}_{F_{i}}(\iota_{1}(p))
\right| >\varepsilon \right)\\
\leq& 
\frac{1}{k^3\varepsilon^4}
\left(\mathbb{I}_{F_i^4}(\iota_p)^4+6\mathbb{I}_{F_{ii}}(\iota_p)\mathbb{I}_{F_i}(\iota_p)^2+\mathbb{I}_{F_i}(\iota_p)^4\right)
+\frac{3(k-1)}{k^3\varepsilon^4}(\mathbb{I}_{F_{ii}}(\iota_1(p))+\mathbb{I}_{F_i}(\iota_p)^2)^2
,\\
&\mathbb{P}\left( 
\left|
\frac1k F_{ij}^{+k}\circ (X_1,\ldots,X_k)- \mathbb{I}_{F_{ij}}(\iota_{1}(p))
\right| >\varepsilon \right)\\
\leq &\frac{1}{k\varepsilon^2}\left(\mathbb{I}_{F_{ij}^2}(\iota_{1}(p))
-\mathbb{I}_{F_{ij}}(\iota_{1}(p))^2\right)
+\frac{k-1}{k\varepsilon^2}\left(\mathbb{I}_{(F_{ij}\circ \pi_1^2)(F_{ij}\circ \pi_2^2) }(\iota_{2}(p))
-
\mathbb{I}_{F_{ij}}(\iota_{1}(p))^2\right).
\end{align*}
\end{enumerate}
\end{theorem}
\begin{proof}
Let $(\mathcal{M}_{q_k,dk}, \theta, T, 0, \psi)$ be an $\exp_{q_k}$-coordinate representation on $\mathcal{M}_{q_k,dk}$ as in~\eqref{qd}.
Then $(\mathcal{M}_{q_k,dk}, \theta, T, 0, \psi)$ is also an  $\exp_{f_{q_k},\chi_{q_k}}$-coordinate representation by the property $\exp_q=\exp_{f_q,\chi_{q}}$.
Since $\det (\beta_{k}(S)S)$ does not depend on $S\in \mathrm{Sym}_+(d,\mathbb{R})$, 
we have 
\[
\mathbb{I}_{\chi_{q_k}}(\iota_{k}(p_q^{v,S}))= 
\mathbb{I}_{\chi_{q_k}}(\rho_{q,k}^{0,I_d}).
\]

Let us  prove Claim~\eqref{791}.
Fix ${\bm x_\ast}\in (\mathbb{R}^d)^k$ and  $p_\ast\in \mathcal{M}_{q,d}^{\tr }$. 
Assume that  $ \iota_{k}(p_\ast)$ maximizes the function  $e_{\bm{x}_\ast} $ on $\mathcal{M}_{q_k,dk}$.
Since $\ln_{q_k}$ is an increasing function, $ \theta(\iota_{k}(p_\ast))$ also maximizes the function 
\[
\ln_{q_k} \circ e_{\bm{x}_\ast} \circ \theta^{-1}
=\langle \cdot, T(\bm{x}_\ast)\rangle-\psi\circ \theta
\quad\text{on }\theta(\mathcal{M}_{q_k,dk}).
\]
By Corollary~\ref{psipos},
$\psi\circ \theta^{-1}$ is strictly convex on $\theta(\mathcal{M}_{q_k,dk})$ 
hence $ \ln_{q_k} \circ e_{\bm{x}_\ast} \circ \theta^{-1}$ is strictly concave   on $\theta(\mathcal{M}_{q_k,dk})$.
Then 
$ \iota_{k}(p_\ast)$  maximizes  $\ln_{q_k} \circ e_{\bm{x}_\ast}$ on an open set $\theta(\mathcal{M}_{q_k,dk})$ 
if and only if 
\[
\frac{\partial}{\partial \theta^i} (\ln_{q_k} \circ e_{\bm{x}_\ast})\big|_{ \iota_{k}(p_\ast)}=0
\quad \text{for }1\leq i\leq N_k,
\]
which is equivalent to 
\[
T_i(\bm{x}_\ast)=\frac{\partial}{\partial \theta^i} \psi\big|_{ \iota_{k}(p_\ast)}
\quad \text{for }1\leq i\leq N_k.
\]
This with Lemma~\ref{lem}\eqref{463} yields 
\[
\mathbb{I}_{T_i\chi_{q_k}}(\iota_{k}(p_\ast))= 
\frac{\partial}{\partial \theta^i} \psi\big|_{ \iota_{k}(p_\ast)}
\cdot \mathbb{I}_{\chi_{q_k}}(\iota_{k}(p_\ast))
=
T_i(\bm{x}_\ast)\cdot\mathbb{I}_{\chi_{q_k}}(\rho_{q,k}^{0,I_d}).
\]
Thus \eqref{thmeq1} holds.
For  $\rho\in \mathcal{P}(\mathcal{L}^{dk})$ satisfying~\eqref{rhoo},
we apply Proposition~\ref{maxcor} to have
\[
D_{f_{q_k},\chi_{q_k}}(\rho, \rho')
=
D_{f_{q_k},\chi_{q_k}}(\rho, \iota_{k} (p_\ast))
+D_{f_{q_k},\chi_{q_k}}(\iota_{k} (p_\ast), \rho')
\quad \text{for }\rho' \in \mathcal{M}_{q_k,dk},
\]
which ensures that $\iota_{k}(p_\ast)$ is a unique minimizer of $D_{f_{q_k},\chi_{q_k}}(\rho,\cdot)$ on $\mathcal{M}_{q_k,dk}$.
These completes the proof of Claim~\eqref{791}.

Let us prove Claim~\eqref{792}.
For $k\in \mathbb{N}$,
set 
\begin{align}\label{rv}
Y_k\coloneqq F_i(X_k)-\mathbb{I}_{F_i}(\iota_1(p)),\quad
Z_k\coloneqq F_{ij}(X_k)-\mathbb{I}_{F_{ij}}(\iota_1(p)),
\qquad \text{for }1\leq i,j\leq d.
\end{align}
For $1\leq m\leq k,$ 
we see that 
\begin{align*}
 \mathbb{I}_{F_i \circ \pi^k_m}(\iota_k(p))
 &=\int_{(\mathbb{R}^d)^k}  F_i(\pi^k_m({\bm x}))  \iota_k(p)({\bm x})d {\bm x}
 =\int_{\mathbb{R}^d}  F_i(x)  \iota_1(p)(x)dx
 =  \mathbb{I}_{F_i}(\iota_1(p)),
 \end{align*}
consequently
\[
\int_{(\mathbb{R}^d)^k}  \left(F_i( \pi^k_m({\bm x}))-\mathbb{I} _{F_i}(\iota_1(p))\right) \iota_k(p)({\bm x})d {\bm x}=0.
\]
Moreover, the change of variables implies
\[
\int_{(\mathbb{R}^d)^k}  
\left(F_i( \pi^k_m({\bm x}))-\mathbb{I} _{F_i}(\iota_1(p))\right)
\left(F_i( \pi^k_{m'}({\bm x}))-\mathbb{I} _{F_i}(\iota_1(p))\right)
 \iota_k(p)({\bm x})d {\bm x}=0
 \quad\text{for }1\leq m'\leq k.
\]
Since $\iota_k(p)$ has a finite fourth moment and $F_{i}^2=F_{ii}$, 
we have 
\begin{align*}
\mathbb{E}(Y_1^4)
&= \int_{\mathbb{R}^d} 
\left(F_i(x)-\mathbb{I}_{F_i}(\iota_p)\right)^4 \iota_1(p)(x)dx\\
&=\mathbb{I}_{F_i^4}(\iota_p)^4+6\mathbb{I}_{F_{ii}}(\iota_p)\mathbb{I}_{F_i}(\iota_p)^2+\mathbb{I}_{F_i}(\iota_p)^4,\\
\mathbb{E}(Y_1^2Y_2^2)
&= \int_{(\mathbb{R}^d)^2} 
\left(F_i(x)-\mathbb{I}_{F_i}(\iota_p)\right)^2
\left(F_i(y)-\mathbb{I}_{F_i}(\iota_p)\right)^2
 \iota_2(p)(x,y)dxdy\\
 &=(\mathbb{I}_{F_{ii}}(\iota_1(p))+\mathbb{I}_{F_i}(\iota_p)^2)^2,
\end{align*}
and similarly
\begin{align*}
\mathbb{E}(Y_1Y_2Y_3Y_4), \mathbb{E}(Y_1^2Y_2Y_3), \mathbb{E}(Y_1^3Y_2)=0.
\end{align*}
Then, for $k\geq 4$,  we derive from Chebyshev's  inequality that 
\begin{align*}
\mathbb{P}\left(\left|\frac1k\sum_{m=1}^k Y_m\right|>\varepsilon  \right)
&\leq 
\frac{1}{k^4\varepsilon^4}
\mathbb{E}\left(\left(\sum_{m=1}^k Y_m\right)^4\right)\\
&=
\frac{1}{k^4\varepsilon^4}
\left(
\sum_{m=1}^k \mathbb{E}\left(Y_m^4\right)
+
3\sum_{m,m'=1, m\neq m'}^k \mathbb{E}\left(Y_m^2 Y_{m'}^2\right)
\right)\\
&=\frac{1}{k^3\varepsilon^4}\mathbb{E}\left( Y_1^4\right)
+
\frac{3(k-1)}{k^3\varepsilon^4}\mathbb{E}\left(  Y_1^{2}Y_{2}^2\right),
\end{align*}
proving the first inequality in Claim~\eqref{792}.
Again,  we apply Chebyshev's inequality to have
\begin{align*}
\mathbb{P}\left( \left|\frac1k\sum_{m=1}^kZ_m\right| >\varepsilon \right)
&\leq \frac{1}{k^2\varepsilon^2} \mathbb{E}\left( \left(\sum_{m=1}^kZ_m \right)^2\right)\\
&=\frac{1}{k^2\varepsilon^2} 
\left(\sum_{m=1}^k
\mathbb{E}\left(Z_m ^2\right)
+
\sum_{m,m'=1,m\neq m'}^k
\mathbb{E}\left(Z_m Z_{m'}\right)
\right)\\
&=\frac{1}{k\varepsilon^2}
\mathbb{E}(Z_1^2)
+\frac{k-1}{k\varepsilon^2}
\mathbb{E}(Z_1Z_2).
\end{align*}
A direct calculation provides 
\begin{align*}
\mathbb{E}(Z_1^2)
&=\int_{\mathbb{R}^d} \left( F_{ij}(x)-\mathbb{I}_{F_{ij}} (\iota_1(p))\right)^2 
\iota_{1}(p)(x)dx 
=
\mathbb{I}_{F_{ij}^2}(\iota_{1}(p))
-\mathbb{I}_{F_{ij}}(\iota_{1}(p))^2,\\
\mathbb{E}(Z_1Z_2)
&=\int_{(\mathbb{R}^d)^2} 
\left( F_{ij}(x)-\mathbb{I}_{F_{ij}} (\iota_1(p))\right)
\left( F_{ij}(y)-\mathbb{I}_{F_{ij}} (\iota_1(p))\right)
\iota_{2}(p)(x)dx dy\\
&=
\mathbb{I}_{(F_{ij}\circ \pi_1^2)(F_{ij}\circ \pi_2^2) }(\iota_{2}(p))
-
\mathbb{I}_{F_{ij}}(\iota_{1}(p))^2,
\end{align*}
implying  the second  inequality in Claim~\eqref{792}

Thus the proof is achieved.
\end{proof}
\begin{proof}[Proof of Theorem~\ref{mainthm3}]
We deduce from  $\mathcal{M}=\mathcal{M}_{q,d}^{I_d}$ that  $\mathcal{M}$ is an identically  repeatable $\exp_q$-family.
Moreover, we can assume $T=(F_i)_{i=1}^d$
since a global $\exp_q$-coordinate representation is uniquely determined up to affine transformation 
by Proposition~\ref{minimal}.
For $\varepsilon>0$, 
it follows from  Theorem~\ref{mainthm33}~\eqref{792} that 
\[
\sum_{k=1}^\infty\mathbb{P}\left( \left|\frac1k F_i^{+k}(X_m)-\mathbb{I}_{F_i}(\iota_1(p))\right|>\varepsilon\right)<\infty.
\]
Then, we apply the Borel--Cantelli lemma to have 
\[
\frac{1}{k}Y_{k}
\xrightarrow{l\to\infty}0
\quad \text{$\mathbb{P}$-a.s.},
\]
which completes the proof.
\end{proof}

\begin{remark}
\begin{enumerate}
\setlength{\leftmargin}{23pt}
\setlength{\leftskip}{-15pt}
\item
We can show  \eqref{lln} for $p\in \mathcal{M}_{q,d}^{\tr}$ and $T=F_i$
using a similar argument as in the proof of Theorem~\ref{mainthm3}.
However, it is unclear whether \eqref{lln} holds for $p\in\mathcal{M}_{q,d}^{\text{tr}}$ and $T=F_{ij}$ or not.
\item
The estimate in Theorem~\ref{mainthm33}~\eqref{792} can be improved.
Indeed, a kind of the weak law of large numbers similar to Theorem~\ref{mainthm33}~\eqref{792}  for $F_1$ with $d=1$ is discussed in ~\cite{EKM}*{Theorem~1.5},
where  a $k^{-1/2}$-Lipschitz function $\Phi\colon\mathbb{R}^{k}\to \mathbb{R}$ is treated and  the  error is estimated above by $C_q (k\varepsilon^2)^{-\delta_q}$ 
with some constant $C_q$ depending on~$q$ and 
\[
\delta_q=\frac{3}{2}+\frac{1}{q-1}>2.
\]
The choice $\Phi=F_i^{+k}$ is possible 
and the estimate is  better (resp.\,worse) than our estimate $C_q (k^2\varepsilon^4)^{-1}$ 
with respect to $k$ (resp.\,$\varepsilon$).
Moreover, in \cite{EKM}*{Proposition~7.5},
a different type of error estimate for $F_{ij}$ with $d=1$ is given.
\setlength{\leftskip}{-15pt}
\item
If we do not pay attention to the marginal condition~\eqref{marginal},
it is possible to define a repetition $(\iota_k)_{k=1}^\infty$ with   
$\alpha_k=\mathrm{id}_I$ and  $b_k= \mathbf{1}_{\mathcal{M}}$ as well as the case of an exponential family.
In such a case, the $(h,\mathrm{id}_I)$-dependence is called  the \emph{$h^\ast$-independence} 
(see for example \cite{FM2012}*{Definition 4}).
 \setlength{\leftskip}{-15pt}
\end{enumerate}
\end{remark}
\begin{ack}
The authors would like to thank Syota Esaki, Daisuke Kazukawa and Eren Mehmet K{\i}ral for fruitful discussions.
The first author was supported in part by JSPS KAKENHI Grant Number 23K03088.
The second author was supported in part by JSPS KAKENHI Grant Numbers 
24H00183, 24K21513.
\end{ack}

\end{document}